\documentclass[11pt]{amsart} 
\usepackage[margin=1in]{geometry}
\usepackage{amssymb}\usepackage[linesnumbered,commentsnumbered,ruled,vlined]{algorithm2e}
\usepackage{xcolor}
\usepackage{tikz}
\usepackage{hyperref}
\usepackage{tikz-cd}
\usetikzlibrary{patterns,external}
\usetikzlibrary{calc}
\usepackage[numbers,sort]{natbib}
\usepackage{paralist}
\usepackage{mathtools}

\definecolor{fondo}{rgb}{0.898,0.996,0.898}

\pgfkeys{/tikz/.cd,
  K/.store in=\K,
  K=1   
   }

\title{Sparse Trace Tests}
\author{Taylor Brysiewicz}
\address{University of Notre Dame, Department of Applied and Computational Mathematics and Statistics, Notre Dame, IN 46556}
\email{tbrysiew@nd.edu}
\author{Michael Burr}
\thanks{Burr was partially supported by grants from the National Science Foundation (CCF-1527193 and DMS-1913119).}
\address{Clemson University, School of Mathematical and Statistical Sciences, 220 Parkway Drive, Clemson, SC 29634}
\email{burr2@clemson.edu}

\newcommand{\C}{{\mathbb{C}}}
\newcommand{\rectangle}{{%
  \ooalign{$\sqsubset\mkern3mu$\cr$\mkern3mu\sqsupset$\cr}%
}}
\newcommand{\R}{{\mathbb{R}}}

\newcommand{\Z}{{\mathbb{Z}}}
\newcommand{\V}{{\mathcal{V}}}

\newcommand{\cA}{{\mathcal{A}}}
\newcommand{\cC}{{\mathcal{C}}}

\newcommand{\cF}{{\mathcal{F}}}
\newcommand{\cL}{{\mathcal{L}}}
\newcommand{\cU}{{\mathcal{U}}}
\newcommand{\cN}{{\mathcal{N}}}

\newcommand{\cB}{{\mathcal{B}}}

\newcommand{\cG}{{\mathcal{G}}}
\newcommand{\mydef}[1]{{\it{\color{blue}#1}}}
\newcommand{\mycomment}[1]{}
\newcommand{\MV}{\text{MV}}
\newcommand{\conv}{\text{conv}}

\newcommand{\TAL}{TAL}

\usepackage[all,cmtip]{xy}

\theoremstyle{definition}
\newtheorem{theorem}{Theorem}
\newtheorem{definition}[theorem]{Definition}
\newtheorem{lemma}[theorem]{Lemma} 
\newtheorem{corollary}[theorem]{Corollary} 
 
\newtheorem{example}[theorem]{Example}
\newtheorem{remark}[theorem]{Remark}
\newtheorem{proposition}[theorem]{Proposition}

\DeclareMathOperator{\res}{Res}
\DeclareMathOperator{\elim}{Elim}
\DeclareMathOperator{\rank}{rk}

\DeclareMathOperator{\defect}{def}
\DeclareMathOperator{\supp}{supp}

\DeclareMathOperator{\offset}{offset}
\DeclareMathOperator{\id}{id}

\begin{document}
\begin{abstract}
We establish how the coefficients of a sparse polynomial system influence the sum (or the trace) of its zeros.  As an application, we develop numerical tests for verifying whether a set of solutions to a sparse system is complete.  These algorithms extend the classical trace test in numerical algebraic geometry.  Our results rely on both the analysis of the structure of sparse resultants as well as an extension of Esterov's results on monodromy groups of sparse systems.
\end{abstract}
\maketitle

\section{Introduction}

The coordinate-wise sum of a finite set of points $S \subseteq \mathbb{C}^n$ is called its {trace}. When $S$ is a subset of a general linear section of an irreducible variety, 
the behavior of its trace as the section moves determines whether $S$ comprises the whole section. The following lemma makes this precise.
\begin{lemma}[{\cite[Theorem 3.6]{TraceTest:2002}}]
\label{lem:tracetest}
Fix an irreducible variety $X \subseteq \mathbb{C}^n$ and a generic pencil of affine linear spaces $L_t$ of complementary dimension. The trace of $X \cap L_t$ moves affine linearly in $t$. Conversely, the trace of any nonempty proper subset of $X \cap L_t$ moves nonlinearly in $t$. 
\end{lemma}
The ({classical}) {trace test} \cite{Multiprojective:2020,TraceTest:2018,TraceTest:2002} is a fundamental algorithm in numerical algebraic geometry \cite{NumericalAlgGeom:1996,IntroNumericalAlgGeom:2005} which is used to verify the affine linear behavior of the trace in Lemma \ref{lem:tracetest}. Our main result is a sparse analogue to Lemma \ref{lem:tracetest}. More precisely, we identify a subset of the coefficients of a sparse polynomial system such that the trace is an affine linear function of this collection of coefficients. Conversely, we show that under simple conditions on the support, the traces of  nonempty incomplete solution sets are nonlinear functions of this collection of coefficients. We use our results to produce what we call sparse trace tests.

\begin{figure}[htpb]
\begin{center}
\begin{tikzpicture}[scale=.35]
\draw[line width=1.5pt] plot[id=trace, raw gnuplot, smooth] function{
f(x,y) = -x**4-13*x**3*y+10*x**3+18*x**2*y**2-10*x**2*y-5*x**2+10*x*y**3+22*x*y**2-10*x*y-4*x-y**3-10*y**2-5*y+6;
set xrange [-8:8];
set yrange [-8:8];
set view 0,0;
set isosample 5000,5000;
set size square;
set cont base;
set cntrparam levels incre 0,0.1,0;
unset surface;
splot f(x,y)};
\draw[dashed] (-8,5.95530726256983) -- (8,-6.73743016759777);
\draw[](-1,-8) -- (7,8);
\draw[](-4,-8) -- (4,8);
\draw[](-7,-8) -- (1,8);
\filldraw[fill=white] (4.00309, 2.00619) circle (.15);
\filldraw[fill=white] (-.073208, -6.14642) circle (.15);
\filldraw[fill=white] (.929176, -4.14165) circle (.15);
\filldraw[fill=white] (3.17294, .345878) circle (.15);
\filldraw[fill=gray] (2.008, -1.984) circle (.15);
\filldraw[fill=white] (-.929974, -1.85995) circle (.15);
\filldraw[fill=white] (-.397977, -.795954) circle (.15);
\filldraw[fill=white] (.250765, .50153) circle (.15);
\filldraw[fill=white] (.517185, 1.03437) circle (.15);
\filldraw[fill=gray] (-.14, -.28) circle (.15);
\filldraw[fill=white] (.194917, 6.38983) circle (.15);
\filldraw[fill=white] (-2.40181, 1.19639) circle (.15);
\filldraw[fill=white] (-2.12884, 1.74233) circle (.15);
\filldraw[fill=white] (-4.81627, -3.63255) circle (.15);
\filldraw[fill=gray] (-2.288, 1.424) circle (.15);
\end{tikzpicture}
\end{center}
\caption{The centroids (gray) of the intersection points (white) of the quartic curve with three parallel lines.  These averages move in a line (dashed).  We use averages instead of sums in this figure to keep the image small. }
\end{figure}
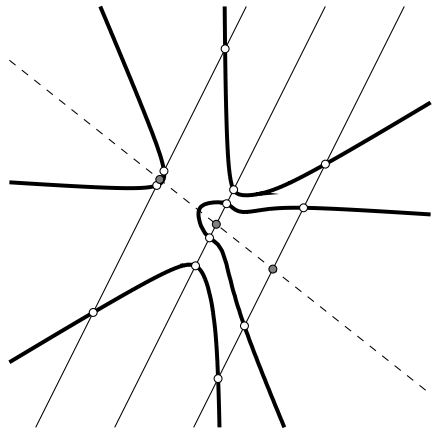

The proof of the forward direction of Lemma \ref{lem:tracetest} rests upon the following elementary result dating back to Newton. The trace of the zeros of $f=c_0+c_1x+c_2x^2+\cdots + c_{d-1}x^{d-1} + c_dx^d \in \C[x]$ equals $-c_{d-1}/c_{d}$, whenever $c_d \neq 0$ and zeros are counted with multiplicity. In particular, the trace of the zeros of a polynomial in one variable is an affine linear function of the coefficient $c_{d-1}$ and does not depend on the coefficients of lower-degree monomials, see Figure~\ref{fig:univariate}.

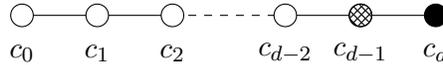
\begin{figure}[htb]
\begin{tikzpicture}
\draw (0,0) -- (1,0) -- (2,0); 
\draw[dashed] (2,0) -- (3.5,0);
\draw (3.5,0) -- (4.5,0) -- (5.5,0);

\filldraw[fill=white] (0,0) circle (.15);
\filldraw[fill=white] (1,0) circle (.15);
\filldraw[fill=white] (2,0) circle (.15);
\filldraw[fill=white] (3.5,0) circle (.15);
\filldraw[fill=white] (4.5,0) circle (.15);
\filldraw[pattern=crosshatch] (4.5,0) circle (.15);
\filldraw[] (5.5,0) circle (.15);

\node at (0,-0.5) {$c_0$};
\node at (1,-0.5) {$c_1$};
\node at (2,-0.5) {$c_2$};
\node at (3.5,-0.5) {$c_{d-2}$};
\node at (4.5,-0.5) {$c_{d-1}$};
\node at (5.5,-0.5) {$c_{d}$};

\end{tikzpicture}
\caption{The support of a univariate polynomial $f$. Identifying each point in the figure with the corresponding coefficient of $f$, the figure displays those which do not affect the trace (white), those which affect the trace affine linearly (cross-hatched), and those which influence the trace nonlinearly (filled).}\label{fig:univariate}
\end{figure}
We generalize Newton's result. Given a tuple $\mathcal A = (A_1,\ldots,A_n)$ of monomial supports $A_i \subseteq \mathbb{Z}^n$, we consider sparse polynomial systems $\mathcal F = (f_1,\ldots,f_n) \in (\mathbb{C}[x_1,\ldots,x_n])^n$ supported on $\mathcal A$. Any such polynomial system is identified with its coefficients in $\mathbb{C}^\mathcal A\vcentcolon=\mathbb{C}^{|\cA|}$.  The polyhedral geometry of $\cA$ controls many aspects of the solutions to $\mathcal F=0$ in the algebraic torus $(\mathbb{C}^\times)^n$ (see \cite{CLS}), and we study the trace of this solution set. As in the univariate case, the trace of these solutions is a rational function of the coefficients of $\cF$. 
Using a simple discrete geometric construction, we identify a large collection of monomials of $\cA$ for which the trace is an affine linear function of the corresponding collection of coefficients.
We illustrate an example in Figure \ref{fig:sparsesupport} where the similarity to Figure \ref{fig:univariate} is immediate.

\begin{figure}[htb]
\begin{tikzpicture}[scale=.7]
\filldraw[fill=gray, fill opacity=0.3] (0,0) -- (2,0) -- (3,1) -- (2,4) -- (1,4) -- (0,2) -- cycle;
\foreach \i in {0,...,4} {\foreach \j in {0,...,5} {\filldraw[color=gray] (\i,\j) circle (.03);}}
\filldraw[fill=white] (0,0) circle (.15);
\filldraw[fill=white] (1,0) circle(.15);
\filldraw[fill=white] (1,2) circle(.15);
\filldraw[pattern=crosshatch] (1,0) circle(.15);
\filldraw (2,0) circle(.15);
\filldraw[fill=white] (1,1) circle (.15);
\filldraw (3,1) circle (.15);
\filldraw (0,2) circle (.15);
\filldraw[fill=white] (1,2) circle (.15);
\filldraw[fill=white] (2,2) circle (.15);
\filldraw[pattern=crosshatch] (2,2) circle (.15);
\filldraw[fill=white] (1,3) circle (.15);
\filldraw[pattern=crosshatch] (1,3) circle (.15);
\filldraw (2,3) circle (.15);
\filldraw[color=white] (1,4) circle (.05);
\filldraw[] (1,4) circle (.15);
\filldraw (2,4) circle (.15);
\node at (2,-1) {Support of $f_1$};

\begin{scope}[shift={(9,0)}]

\foreach \i in {0,...,3} {\foreach \j in {0,...,4} {\filldraw[color=gray] (\i,\j) circle (.03);}}
\filldraw[fill=gray, fill opacity=0.3] (0,0) -- (2,0) -- (2,3) -- (0,3) -- cycle;
\filldraw[fill=white] (0,2) circle (.15);
\foreach \i in {0,...,1} {\foreach \j in {0,...,3} {\filldraw[fill=white] (\i,\j)  circle (.15);}}
\foreach \i in {1,...,1} {\foreach \j in {0,...,3} {\filldraw[pattern=crosshatch] (\i,\j)  circle (.15);}}
\foreach \i in {2,...,2} {\foreach \j in {0,...,3} {\filldraw (\i,\j)  circle (.15);}}
\filldraw[] (1,3) circle (.15);
\filldraw[] (0,3) circle (.15);
\filldraw[pattern=crosshatch] (0,2) circle (.15);
\node at (1.5,-1) {Support of $f_2$};
\end{scope}
\end{tikzpicture}
\caption{The support of $\cF=(f_1,f_2)$.  Identifying each point in the figure with the corresponding coefficient in $\cF$, our results determine a collection of those which do not affect the trace (white), a collection of those which may affect the trace affine linearly (cross-hatched), and a collection of those which may influence the trace nonlinearly (filled). 
}\label{fig:sparsesupport}
\end{figure}
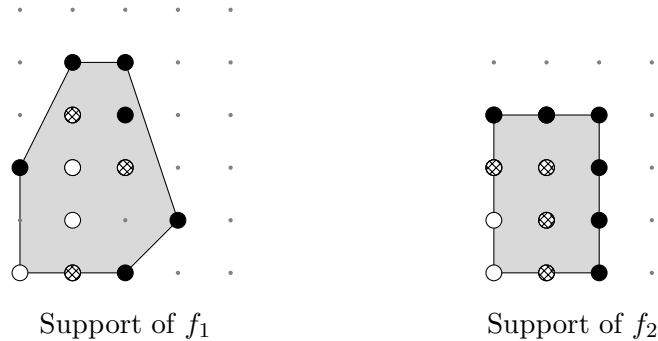

The backward direction of Lemma \ref{lem:tracetest} is more sophisticated. The proof in \cite{TraceTest:2018} relies on the fact that the monodromy group of the branched cover $\pi:\{(x,t) \in \C^n_x \times \C_t \,\, |\,\, x \in X \cap L_t\} \rightarrow \C_t$ is the full symmetric group. In the sparse setting, the relevant branched cover is the map
$$\pi_{\mathcal A}: \{(x,\mathcal F) \in (\mathbb{C}^\times)^n \times \mathbb{C}^{\mathcal A}\,\, |\,\, \mathcal F(x)=0\} \to \mathbb{C}^{\mathcal A},$$
restricted to those sparse systems where some subset of the coefficients are generic and fixed, and the rest vary.  Given a set $\cA$ of supports, we provide a simple condition on the  varying coefficients which guarantees that the monodromy group of this restriction is the full symmetric group.

The classical and sparse trace tests are instances of what we call {completeness tests}, which may be used to decide if a nonempty subset of a zero-dimensional variety is proper.  When implemented using numerical computations, such tests return the correct answer almost surely, and we call them numerical completeness tests.   For the classical trace test, the zero-dimensional polynomial system is the pair consisting of the equations for an irreducible variety and a generic linear system of complementary dimension. We expect  completeness tests to be developed beyond the sparse case; for example, to generalizations of witness sets in numerical algebraic geometry \cite{Multiprojective:2020, Sottile2020}.

Our results, although motivated by applications in numerical algebraic geometry, have consequences in symbolic algebra. For example, our results show how to reduce a polynomial system while preserving its trace. This reduction leads to smaller systems which improves the efficiency of both numerical and symbolic computations.

\subsection{Overview}

Since our main results can be stated, understood, and applied without proof, we present them in Section \ref{sec:BackgroundAndMainResults} along with the relevant ideas and notation. There, we present two algorithms, which we call sparse trace tests. The detailed background and proofs reside in Sections \ref{sec:CombinatorialApproximationsOfTheSupport} and \ref{sec:Monodromy}. In particular, we use sparse resultants in Section \ref{sec:CombinatorialApproximationsOfTheSupport} to determine coefficients of a sparse system which do not appear in a formula for the trace. In Section \ref{sec:Monodromy} we investigate the monodromy groups of restrictions of sparse polynomial systems to establish when traces of incomplete solution sets move nonlinearly. Section \ref{sec:Monodromy} also addresses exceptional cases for which the sparse trace tests cannot be applied. In Section \ref{sec:Examples}, we present a gallery of examples showcasing  applications of our theory and algorithms.

\section{Notation, background, and main results}
\label{sec:BackgroundAndMainResults}
We introduce the required background and notation pertaining to sparse polynomial systems and state our main results. We present two sparse trace tests (Algorithms \ref{algo:sparsetracetest} and  \ref{algo:constantsparsetracetest}), and outline a proof of correctness for these algorithms, relying upon results from Sections \ref{sec:CombinatorialApproximationsOfTheSupport} and \ref{sec:Monodromy}.

\subsection{Sparse polynomial systems}
Let $\mydef{e_j}$ denote the $j$-th standard coordinate vector of the integer lattice $\Z^n$ and $\mydef{{\bf 0}}$ the origin of $\Z^n$. 
For a finite subset $\mydef{A}\subseteq \Z^n$, we consider (Laurent) polynomials of the form $$\mydef{f(x)}=f(x_1,\dots,x_n) = \sum_{\alpha \in A} \mydef{c_{\alpha}}x_1^{\alpha_1}\cdots x_n^{\alpha_n} = \sum_{\alpha \in  A} c_\alpha \mydef{x^\alpha} \in \mydef{\mathbb{C}[x^{\pm \textrm{1}}]}\vcentcolon=\mathbb{C}[x_1,\ldots,x_n,x_1^{-1},\ldots,x_n^{-1}],$$
where $c_\alpha\in\C$.   We say that  $f$ is \mydef{supported on} $A$ and its \mydef{support} is $\mydef{\supp(f)}=\{\alpha \in A \mid c_{\alpha} \neq 0\}$.  Identifying $f$ with its coefficients $\{c_{\alpha}\}_{\alpha \in A}$, we write $f \in \mydef{ \mathbb{C}^{A}}\vcentcolon=\mathbb{C}^{|A|}$. We define $\mydef{L[A]}$ to be the lattice generated by all differences of points in $A$. Given any sublattice $L$ of $\Z^n$, its rank is $\mydef{\rank{(L)}}$, and we define $\mydef{\rank(A)} \vcentcolon = \rank(L[A])$. When  $L$ is full rank, we write $\mydef{[\Z^n:L]}$ for the index of $L$ in $\Z^n$.

For a fixed collection $\mydef{\mathcal A} = (A_1,\ldots,A_N)$, a tuple $\mydef{\cF} = (f_1,\dots,f_N)$ of Laurent polynomials in $\C[x^{\pm 1}]$ is called a \mydef{sparse polynomial system supported on} $\cA$ whenever each $f_i$ is supported on $A_i$. In this case, we write $f_i = \sum_{\alpha \in A_i} c_{i,\alpha} x^\alpha$. As with a single polynomial, we identify $\cF$ with its coefficients in $\mydef{\mathbb{C}^{\cA}}\vcentcolon=\mathbb{C}^{A_1} \times \cdots \times \mathbb{C}^{A_N} $. We apply set operations to collections of supports coordinate-wise. For example, given $\cB=(B_1,\ldots,B_N)$ and $\cC=(C_1,\ldots,C_N)$, we write $\mydef{\cB \cup \cC}$ for $(B_1 \cup C_1,\ldots,B_N \cup C_N)$. When $\cA = \cB \sqcup \cC$, we may decompose $\cF = \mydef{\cF_{\cB} + \cF_{\cC}} \in \mydef{\C^{\cB} \times \C^{\cC}}=\C^{\cA}$ in order to refer to the coefficients of $\cF$ indexed by $\cB$ and $\cC$ individually. We define $\mydef{L[\cA]}$ to be the lattice generated by $\bigcup_{i=1}^n L[A_i]$. When each $A_i$ has full rank, we say that $\cA$ is \mydef{abundant}.

The set of zeros in $\C^n$ of a polynomial system is denoted by $\mydef{\V(\cF)}$. In the sparse setting, we consider zeros of $\cF \in \mathbb{C}^{\cA}$ in the \mydef{algebraic torus} $\mydef{(\C^\times)^n} \vcentcolon = \{x \in \C^n \mid x_i \neq 0 \text{ for } i=1,\ldots,n\}$, and so we decorate the notation $\V$ as 
$\mydef{\V^\times(\cF)} \vcentcolon= \{x \in (\C^\times)^n \mid f_i(x) = 0 \text{ for } i=1,\ldots,n\}$.

To refer to the entire family of sparse polynomial systems supported on $\cA$, we replace the coefficients $c_{i,\alpha}$ with parameters $\mydef{u_{i,\alpha}}$ and write $\mydef{\cF(u)} \in \C[u_{i,\alpha}][x^{\pm 1}]$. This system is called the \mydef{universal polynomial system over $\cA$}.
The following \mydef{incidence variety} encodes the structure of the family of polynomial systems $\mathbb{C}^{\cA}$ and their zeros: 
$$\mydef{X_{\cA}}\vcentcolon= \{(x,\cF) \mid x \in \V^\times(\cF), \cF \in \C^{\cA}\} \subseteq (\C^\times)^n \times \C^{\cA}.$$
In other words, $X_{\cA}$ is the zero set of $\cF(u)$ in $(\C^\times)^n \times \C^{\cA}$.
Any particular solution set $\V^\times(\cF)$ is naturally identified with the fibre over $\cF$ with respect to the projection $\mydef{\pi_{\cA}}$ on the second factor. The fibre over a point $x\in(\C^\times)^n$ with respect to the projection $\mydef{\pi_{\V}}$ on the first factor is identified with all sparse systems having $x$ as a solution.

We say that $\cA$ is a \mydef{square} set of supports when $N=n$. Similarly, systems supported on square sets of supports are called \mydef{square systems}. When $\cA$ is square, its mixed volume $\mydef{\MV(\cA)}$ is defined to be the mixed volume of the convex hulls $\{\mydef{\conv(A_i)}\}_{i=1}^n$. 
Unless otherwise stated, $\cA$  refers to a \emph{square set of supports with} $\MV(\cA)>0$. We repeat these assumptions in all results for completeness.

The discrete geometry of $\cA$ controls much of the complex geometry of $X_\cA$. In particular, for square systems, the following theorem relates the mixed volume of $\cA$ to the cardinality of a generic fibre of $\pi_{\cA}$.
\begin{theorem}[BKK \cite{Bernstein,Kushnirenko}]\label{thm:BKK}
Let $\cA$ be a square set of supports with $\MV(\cA)>0$.  The cardinality of $\V^{\times}(\cF)$ is $\MV(\cA)$ for all $\cF$ outside of a Zariski closed subset of $\mathbb{C}^{\cA}$. 
\end{theorem}
Geometrically, Theorem \ref{thm:BKK} states that when $\MV(\cA)>0$, the map $\pi_{\cA}$ is a \mydef{branched cover} of degree $\MV(\cA)$. The Zariski closed subset of $\C^{\cA}$ for which the fibres do not have cardinality $\MV(\cA)$ is called the \mydef{branch locus} of $\pi_{\cA}$. 
We call a system $\cF\in\C^\cA$ \mydef{Bernstein-generic} whenever it is not in this branch locus. Bernstein-generic systems have exactly $\MV(\cA)$-many isolated zeros, each of multiplicity one, in the algebraic torus. More generally, systems which are not Bernstein-generic have at most $\MV(\cA)$-many isolated zeros in the algebraic torus, counted with multiplicity.

Restricted to the complement of the branch locus, the map $\pi_{\cA}$ is topologically a $\MV(\cA)$-to-one covering space. Consequently, a loop $\gamma:[0,1] \to \C^{\cA}$ based at $\cF \in \C^{\cA}$ which avoids the branch locus may be uniquely lifted to $\MV(\cA)$-many paths in $X_{\cA}$. These paths connect points in $\pi_{\cA}^{-1}(\cF)$, inducing a permutation on the fibre. The collection of all such permutations form a subgroup $\mydef{G(\pi_{\cA})}$ of the symmetric group on $\pi_{\cA}^{-1}(\cF)$.  The group ${G(\pi_{\cA})}$ encodes global information about the symmetries of the branched cover and is called the \mydef{monodromy group}  of $\pi_{\cA}$. It does not depend on the base point $\cF$ and is well-defined up to relabeling the points of $\pi_{\cA}^{-1}(\cF)$. 

\subsection{Traces}
Given a subset $S \subseteq \C^n$, its coordinate-wise sum $\mydef{\Sigma(S)} \vcentcolon= (\Sigma_1(S),\ldots,\Sigma_n(S))$ is called its \mydef{trace}. 
In the numerical algebraic geometry literature (see, for example, \cite{TraceTest:2002}), the coordinate-wise \emph{average} $\mydef{\mu(S)}\vcentcolon=\mydef{(\mu_1(S),\dots,\mu_n(S))} = \frac{\Sigma(S)}{|S|}$ of $S$ has been traditionally called its trace. To avoid ambiguity, we call $\mu(S)$ the \mydef{centroid} of $S$.

Without loss of generality, \emph{we focus on the first coordinate} $\Sigma_1(\V^\times(\cF))$. Analogous statements about the other coordinates of $\Sigma(\V^\times(\cF))$ may be gleaned from our results by the interested reader.

The authors of \cite{DAndreaJeronimo:2008} show that $\Sigma_1(\V^\times(\cF(u)))$ is a rational function of the coefficients of $\cF(u)$, expressed in terms of sparse resultants (see Section \ref{sec:CombinatorialApproximationsOfTheSupport} for additional details). We assume that this fraction has been reduced so that the numerator and denominator do not share any nontrivial factors. This formula involves some, but not all, of the coefficients of $\cF(u)$. We distinguish monomials in $\cA$ depending on how they appear in the formula for $\Sigma_1(\V^\times(\cF(u)))$.

\begin{definition}
\label{def:partitionmonomials}
Given a square set of supports $\cA$ with $\MV(\cA)>0$, we define the following:
\begin{itemize}
\item The \mydef{unnecessary support} $\mydef{\cU} \subseteq \cA$ consists of those monomials whose coefficients do not appear in the formula for $\Sigma_1(\V^\times(\cF(u)))$.  The complement $\mydef{\mathcal N} =\cA \backslash \cU$ of $\cU$  is the \mydef{necessary support}.
\item The \mydef{affine linear support} $\mydef{\mathcal {AL}} \subseteq \cN \subseteq \cA$ consists of those monomials in the necessary support whose coefficients only appear in the numerator, and do so to degree one.
\item The \mydef{nonlinear support} $\mydef{\cN\cL} \subseteq \cN \subseteq \cA$, consists of all other necessary monomials, that is, $\cN\cL = \cN \backslash \cA\cL$.
\end{itemize}
\end{definition}
These sets partition $\cA$ as $\cA=\cU \sqcup \cN=\cU \sqcup (\cA\cL \sqcup \cN\cL)$.  Given two sparse systems $\cF,\cG \in \C^{\cA}$ and a subset $\cC \subseteq \cA$, we write $\cF \mspace{\thickmuskip}\mydef{\approx_{\cC}} \mspace{\thickmuskip}\cG$ when $\cF_{\cC} = \cG_{\cC}$, that is, the coefficients of $\cF$ and $\cG$ agree over the support $\cC$.

\begin{definition}
\label{def:AffineLinearGroup}
A subset $\cB \subseteq \cA$ is called \mydef{trace-affine-linear (\TAL)} if the trace $\Sigma_1(\V^\times(\cF))$ is an affine linear function of the coefficients indexed by $\cB$.
\end{definition}

\begin{remark} Any singleton consisting of an element of $\mathcal U \cup \cA\cL$ is \TAL. Moreover, by definition, any \TAL\ subset must be contained in $\mathcal U \cup \cA\cL$. However, the converse is not true: if two coefficients $c_{\alpha},c_{\beta}$ appear to degree one in the same monomial in the numerator of the formula for $\Sigma_1(\V^\times(\cF(u)))$, then the set $\{\alpha,\beta\}$ is not \TAL\ despite being contained in $\cU \cup \cA\cL$.
\end{remark}

The following lemma provides the analogue of the forward direction of Lemma \ref{lem:tracetest}. It is a direct consequence of Definitions \ref{def:partitionmonomials} and  \ref{def:AffineLinearGroup}.
\begin{lemma}\label{lem:preciseForwardSparse}
Suppose $\cA$ is a square set of supports with $\MV(\cA)>0$ and $\cB \subseteq \cA$ is a \TAL\ subset of $\cA$.  Fix Bernstein-generic $\cF,\cG \in \C^{\cA}$ such that $\cF \approx_{\cA\backslash \cB} \cG$ (respectively, $\cF \approx_{\cN} \cG$). Then the trace $\Sigma_1$ of $\V^\times(t\cF+(1-t)\cG)$ is an affine linear (respectively, constant) function of $t$, restricted to the set of $t$'s where $t\cF+(1-t)\cG$ is Bernstein-generic.
\end{lemma}

\begin{example}\label{ex:sparsetrace}
Let 
\begin{equation}\label{eq:ex:F}
\cF=\begin{Bmatrix}
f_1:\hspace{0.55 in} 3x^{2}y^{4}+2x^{2}y^{3}-xy^{4}+x^{3}y+5x^{2}y^{2}+xy^{3}+2x^{2}+4y^{2}+9x\\
f_2: 4x^{2}y^{3}+x^{2}y^{2}+8xy^{3}-4x^{2}y-2xy^{2}+3y^{3}+4x^{2}+5xy+4x-y-9
\end{Bmatrix},\hspace{0.53 in}
\end{equation}
and 
\begin{equation}\label{eq:ex:G}
\cG=\begin{Bmatrix}
g_1: 3x^{2}y^{4}+2x^{2}y^{3}-3xy^{4}+x^{3}y+3x^{2}y^{2}-2xy^{3}-3xy^{2}+2x^{2}+xy+y^{2}+4x-5\\
g_2:\hspace{0.15in} 4x^{2}y^{3}+x^{2}y^{2}+3xy^{3}-4x^{2}y-4xy^{2}+5y^{3}+4x^{2}+3xy-4y^{2}-x-2y-5
\end{Bmatrix}.
\end{equation}
These systems are supported on the monomials depicted in Figure \ref{fig:sparsesupport}.

The polynomials $f_1$ and $g_1$ agree on the coefficients of $\{x^2y^4, x^2y^3, x^3y,x^2\}$, as do $f_2$ and $g_2$ with respect to the coefficients of $\{x^2y^3,x^2y^2,x^2y,x^2\}$.  We show in Example \ref{eq:ex:FContinued} that the complement of this collection of coefficients is \TAL\ with respect to $x$.  In Table \ref{table:sparseaffine}, we calculate $\Sigma_1$ and $\Sigma_2$ for $t\cF+(1-t)\cG$ for several values of $t$.  Through direct inspection, we see that $\Sigma_1$ appears to be an affine linear function of $t$, whereas $\Sigma_2$ is not.  Therefore, by Lemma \ref{lem:preciseForwardSparse}, we conclude that the complement of this set of coefficients is not \TAL\ with respect to $y$.
\end{example}

\begin{table}[hbt]
\begin{tabular}{c|cccccc}
$t$&0&1&2&3&4&5\\\hline
$\Sigma_1$&3.922&-0.578&-5.078&-9.578&-14.078&-18.578\\
$\Sigma_2$&-0.200&-0.523&-8.135&5.772&1.974&1.236 
\end{tabular}
\medskip
\caption{The traces for $\V^\times(t\cF+(1-t)\cG)$ where $\cF$ and $\cG$ are Systems (\ref{eq:ex:F}) and (\ref{eq:ex:G}) in Example~\ref{ex:sparsetrace}.  
}\label{table:sparseaffine}
\end{table}

\subsection{Sparse trace tests}
We first recall the classical trace test \cite{TraceTest:2018} since it serves as a model for the sparse trace test.  The input to the algorithm is an irreducible variety $X$, a general family of parallel linear spaces $L_t$ of complementary dimension, and a subset $S\subseteq X\cap L_0$.  The algorithm uses \mydef{homotopy continuation} to construct an analytic continuation $S_t\subseteq X\cap L_t$ of $S=S_0$ as $t$ varies.  This process tracks the points of $S$ to points of $X\cap L_t$ for generic $t\in\C$.  The values of $\Sigma(S_0)$, $\Sigma(S_{1/2})$, and $\Sigma(S_1)$ are compared, and, after appealing to Lemma \ref{lem:tracetest}, these three traces are collinear if and only if $S=X\cap L_0$.

A \mydef{completeness test} is an algorithm whose input is a zero-dimensional polynomial system $\cF$ and a nonempty subset $S$ of its zeros.  It returns the output \texttt{pass} if $S$ is \mydef{complete}, that is, if $S=\V(\cF)$. It returns \texttt{fail} if $S\neq \V(\cF)$. The main application of the classical trace test is to decide whether a nonempty subset $S$ of a general linear section $X\cap L$ of an irreducible variety is complete.   In this sense, the classical trace test is a completeness test for irreducible varieties.

Following the approach of the classical trace test, we introduce a \mydef{sparse trace test} in Algorithm~\ref{algo:sparsetracetest}.  Using the stronger conditions of Lemma \ref{lem:preciseForwardSparse}, we describe a second sparse trace test, Algorithm~\ref{algo:constantsparsetracetest},  which we call the \mydef{constant sparse trace test}. We refer to these algorithms as the \mydef{sparse trace tests}.

\begin{algorithm}[tb]
\SetKwIF{If}{ElseIf}{Else}{if}{then}{elif}{else}{}%
\DontPrintSemicolon
\SetKwProg{Sparse trace test}{Sparse trace test}{}{}
\LinesNotNumbered
\KwIn{$\bullet$ A collection of supports $\cA$ in $\Z^n$ with $\MV(\cA)>0$ and $L[\cA]= \Z^n$\\ \hspace{0.49in}  $\bullet$ A Bernstein-generic polynomial system $\cF \in \C^{\cA}$ \\ \hspace{0.49in}  $\bullet$ $ \emptyset \neq S \subseteq \V^\times(\cF)$ \\ \hspace{0.49in} $\bullet$ A collection $\cB\subseteq \cA$ which is \TAL\ and abundant.
 }
\KwOut{If $S = \V^\times(\cF)$, then \texttt{pass}, else \texttt{fail}.} 
\nl Construct $\cG \in \C^{\cB} \times \cF_{\cA \backslash \cB}$ by choosing generic $\cG_{\cB} \in \C^{\cB}$. \;
\nl Use homotopy continuation to follow the analytic continuation $S_t\subseteq \V^\times(t\cF+(1-t)\cG)$ of $S=S_1$ to $t= 1/2$ and $t= 0$.\;
\nl Compute $\Sigma_1(S_0)$, $\Sigma_1(S_{1/2})$, and $\Sigma_1(S_1)$.\;
\nl \If{$(0,\Sigma_1(S_0))$, $(1/2,\Sigma_1(S_{1/2}))$ and $(1,\Sigma_1(S_1))$ are collinear}{
 \nl  \Return{} \texttt{pass}\;
 }\nl\Else{\nl\Return{} \texttt{fail}}
\caption{Sparse trace test \label{algo:sparsetracetest}}
\end{algorithm} 
\begin{algorithm}[b]
\SetKwIF{If}{ElseIf}{Else}{if}{then}{elif}{else}{}%
\DontPrintSemicolon
\SetKwProg{Constant sparse trace test}{Constant sparse trace test}{}{}
\LinesNotNumbered
\KwIn{$\bullet$ A collection of supports $\cA$ in $\Z^n$ with $\MV(\cA)>0$ and $L[\cA]=\Z^n$\\ \hspace{0.49in} $\bullet$ A Bernstein-generic polynomial system $\cF \in \C^{\cA}$\\ \hspace{0.49in}  $\bullet$ $ \emptyset \neq S \subseteq \V^\times(\cF)$ \\ \hspace{0.49in} $\bullet$ A collection $\cB \subseteq \cU$ which is abundant.
 }
\KwOut{If $S = \V^\times(\cF)$, then \texttt{pass}, else \texttt{fail}.} 
\nl Construct $\cG \in \C^{\cB} \times \cF_{\cA \backslash \cB}$ by choosing generic $\cG_{\cB} \in \C^{\cB}$.  \;
\nl Construct $|S|$-many solutions $S' \subseteq \V^\times(\cG)$.\;
\nl \If{$\Sigma_1(S)=\Sigma_1(S')$}{
 \nl  \Return{} \texttt{pass}\;
 }\nl\Else{\nl\Return{} \texttt{fail}}
\caption{Constant sparse trace test \label{algo:constantsparsetracetest}}
\end{algorithm}

By Lemma \ref{lem:preciseForwardSparse}, the sparse trace tests return \texttt{pass} when $S=\V^\times(\cF)$, even when $\cB$ is not abundant or $L[\cA] \neq \Z^n$. Thus, they return \texttt{fail} only when $S \neq \V^\times(\cF)$. However, from a practical point of view, even this one-sided usage requires \emph{a priori} knowledge of some subset $\cB\subseteq \cA$ which is either \TAL\, or contained in the unnecessary support.  In Section \ref{secsec:approximatingthesupport}, we give simple discrete geometric descriptions of sets satisfying these relationships. The conditions that $\cB$ is abundant and $L[\cA]=\Z^n$ make the algorithm two-sided and are easy to check. In Section \ref{secsec:incompletesolutionsets} we justify the inclusion of these additional conditions on $\cB$.

\subsection{Approximating the support}\label{secsec:approximatingthesupport} 
The use of either sparse trace test requires valid $\cA$ and $\cB$ as input. It is straightforward to check the conditions that $\cB$ is abundant and $L[\cA]=\Z^n$, but finding a candidate for $\cB$ is, \textit{a priori}, not obvious. We provide easy-to-compute subsets of $\cA$ which are \TAL\ or contained in $\cU$, which may be used for $\cB$.

Let $A\subseteq\Z^n$ be finite.  The set of \mydef{$k$-offset} points of $A$ in the \emph{first coordinate} is defined to be
$$
\mydef{\offset(A,k)}=\{\alpha\in A \,\mid\, \alpha+(k+\varepsilon) e_1\not\in\conv(A)\text{ for all }\varepsilon>0\}.
$$
These sets are increasing, that is, if $k<k'$, then $\offset(A,k)\subseteq \offset(A,k')$.  We note that $\offset(A,0)$ consists of the points of $A$ which maximize a linear functional $x \mapsto \langle x, \omega \rangle$ for some $\omega \in \R^n$ with $\omega_1>0$.  For a collection $\cA$, we define $\offset(\cA,k)$ coordinate-wise.

\begin{figure}[htb]
\begin{tikzpicture}[scale=0.75]

\draw[dashed] (-0.5,0) -- (1.5,0) -- (2.5,1) -- (1.5,4) -- (0.5,4) -- (-0.5,2) -- cycle;
\filldraw[fill=gray, fill opacity=0.3] (0,0) -- (2,0) -- (3,1) -- (2,4) -- (1,4) -- (0,2) -- cycle; 
\foreach \i in {-1,...,4} {\foreach \j in {-1,...,5} {\filldraw[color=gray] (\i,\j) circle (.03);}}
\filldraw[fill=white] (0,0) circle (.15);
\filldraw[color=white] (1,0) circle (.15);
\filldraw[pattern=crosshatch] (1,0) circle(.15);
\filldraw (2,0) circle(.15);
\filldraw[fill=white] (1,1) circle (.15);
\filldraw (3,1) circle (.15);
\filldraw[fill=white] (0,2) circle (.15);
\filldraw[fill=white] (1,2) circle (.15);
\filldraw[fill=white] (2,2) circle (.15);
\filldraw[pattern=crosshatch] (2,2) circle (.15);
\filldraw[fill=white] (1,3) circle (.15);
\filldraw (2,3) circle (.15);
\filldraw[color=white] (1,4) circle (.15);
\filldraw[pattern=crosshatch] (1,4) circle (.15);
\filldraw (2,4) circle (.15);
\node at (2,-2.5) {$(a)$};
\begin{scope}[shift={(9,0)}]
\draw[dashed] (0,-0.5) -- (2,-0.5) -- (3,0.5) -- (2,3.5) -- (1,3.5) -- (0,1.5) -- cycle;
\filldraw[fill=gray, fill opacity=0.3] (0,0) -- (2,0) -- (3,1) -- (2,4) -- (1,4) -- (0,2) -- cycle; 
\foreach \i in {-1,...,4} {\foreach \j in {-1,...,5} {\filldraw[color=gray] (\i,\j) circle (.03);}}
\filldraw[fill=white] (0,0) circle (.15);
\filldraw[fill=white] (1,0) circle(.15);
\filldraw[fill=white] (2,0) circle(.15);
\filldraw[fill=white] (1,1) circle (.15);
\filldraw (3,1) circle (.15);
\filldraw (0,2) circle (.15);
\filldraw[fill=white] (1,2) circle (.15);
\filldraw[fill=white] (2,2) circle (.15);
\filldraw[color=white] (1,3) circle (.15);
\filldraw[pattern=crosshatch] (1,3) circle (.15);
\filldraw[color=white] (2,3) circle (.15);
\filldraw[pattern=crosshatch] (2,3) circle (.15);
\filldraw (1,4) circle (.15);
\filldraw (2,4) circle (.15);
\node at (2,-2.5) {$(b)$};
\end{scope}
\end{tikzpicture}
\caption{The sets $\offset(A,.5)$ and $\offset(A,1)$ in the $(a)$ first and $(b)$ second coordinates along with the convex hull of $A$.  The points of $\offset(A,.5)$ are filled, the points of $\offset(A,1)\backslash\offset(A,0.5)$ are cross-hatched, and the remaining points are white.  The shifted polytope illustrates the behavior for multiple points at once.}
\end{figure}
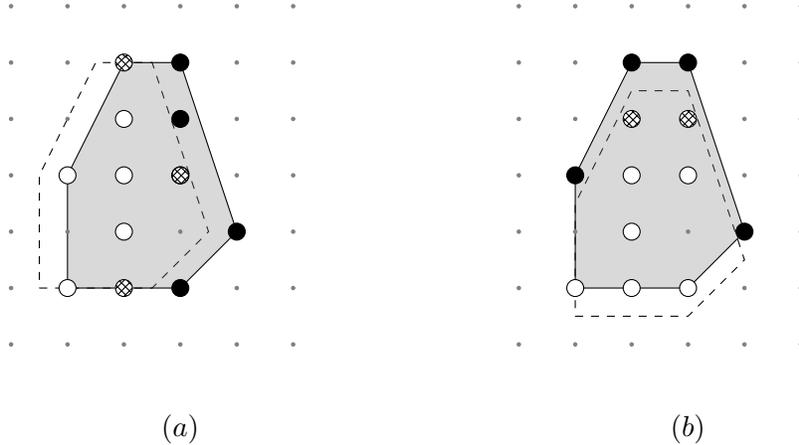

The following specialization of Theorem \ref{thm:combinatorialclassification} provides discrete geometric constructions of subsets of $\cA$ which are either \TAL\ or contained in $\cU$.

\begin{lemma}
\label{lem:ApproximationOfNecessary}
Let $\cA=(A_1,\dots,A_n)$ be a collection of supports in $\Z^n$ such that $\MV(\cA)>0$. Then,
\begin{itemize}
\item $\cA \backslash \offset(\cA,0.5)$ is \TAL\ and
\item $\cA \backslash \offset(\cA,1) \subseteq \cU$.
\end{itemize}
\end{lemma}

\begin{example}
\label{eq:ex:FContinued}
We continue the discussion of Systems (\ref{eq:ex:F}) and (\ref{eq:ex:G}) from Example \ref{ex:sparsetrace}. Since the coefficients of $\cF$ and $\cG$ which agree contain $\offset(\cA,0.5)$ with respect to the first coordinate,  Lemmas \ref{lem:preciseForwardSparse}  and  \ref{lem:ApproximationOfNecessary} imply that the trace $\Sigma_1(\V^\times(t\cF + (1-t)\cG))$ is an affine linear function in the remaining coefficients. On the other hand, since the coefficients of $f_1$ and $g_1$ corresponding to the monomials $\{y^2,xy^4,x^2y^4\}$ disagree, Lemma \ref{lem:ApproximationOfNecessary} does not apply with respect to the second coordinate. Table \ref{table:sparseaffine} shows the linearity of the trace in the first coordinate and nonlinearity of the trace in the second coordinate.
\end{example}

\subsection{Incomplete solution sets}
\label{secsec:incompletesolutionsets}
We explain how the conditions $\cB$ is abundant and $L[\cA]=\Z^n$, on $\cA$ and $\cB$, guarantee that the output \texttt{pass} from the sparse trace tests implies the equality $S = \V^\times(\cF)$. 
Our argument relies on results about the monodromy group $G(\pi_{\cA})$ detailed in Section \ref{sec:Monodromy}, and is structured as follows: Suppose that $\cF$ is Bernstein-generic, and $\cG$ in $\C^{\cB} \times \cF_{\cA\backslash\cB}$ is generic.  When the monodromy group of $\pi_{\cA}$ restricted to the preimage of the line containing $\cF$ and $\cG$ is the full symmetric group, there is a loop which induces the transposition swapping $s \in S$ and $s' \in \V^\times(\cF)\backslash S$. When the points in $\V^\times(\cF)$ have distinct first coordinates, $\Sigma_1(S) \neq \Sigma_1(S \cup \{s'\} \backslash \{s\})$. In these cases, since the trace of the analytic continuation of $S \neq \V^\times(\cF)$ is continuous along the corresponding monodromy loop, the traces along this loop cannot lie on a line.

\begin{theorem}
Let $(\cA,\cF,S,\cB)$ be the input to Algorithm \ref{algo:sparsetracetest} or \ref{algo:constantsparsetracetest}. Then the algorithm returns \texttt{pass} if and only if $S=\V^\times(\cF)$.
\end{theorem}
\begin{proof}
If $\cB$ is \TAL\ (respectively, $\cB \subseteq \cU$), then Lemma \ref{lem:preciseForwardSparse} implies that Algorithm \ref{algo:sparsetracetest} (respectively, Algorithm \ref{algo:constantsparsetracetest}) returns \texttt{pass} when $S=\V^\times(\cF)$ is given as input. It remains to be shown that the algorithms return \texttt{fail}, almost surely, when $S \neq \V^\times(\cF)$.

The conditions on the input to Algorithms \ref{algo:sparsetracetest} or \ref{algo:constantsparsetracetest} imply that the monodromy group $\pi_{\cA}$ restricted to $\C^{\cB} \times \cF_{\cA \backslash \cB}$ is the full symmetric group by Theorem \ref{thm:monodromy}. Zariski showed that the further restriction to $t \cF+(1-t) \cG$ preserves the monodromy group \cite{Zariski:1937}. Theorem \ref{lem:distinctfirstcoordinates} implies that the first coordinates of points in $\V^\times(\cF)$ are distinct. By the argument above the statement of the theorem, the algorithms return \texttt{fail}.
\end{proof}

\begin{remark}
The condition that $\cB$ is abundant simplifies the proofs in Section \ref{sec:Monodromy}. Certainly, less restrictive conditions also suffice, as illustrated by the the classical trace test. Finding precise conditions is left for future research.
\end{remark}

Section \ref{sec:sec:LacunaryTriangular} details why the condition $L[\cA]=\Z^n$ is necessary. The following example illustrates one type of issue that arises when $L[\cA] \neq \Z^n$.

\begin{example}
Suppose that $\cF=(f_1,f_2)$ is a Bernstein-generic system supported on the support $\cA=(A,A)$ where $A=\{(0,0), (2,0), (4,0), (3,1), (0,2), (2,2)\}$, as depicted in Figure \ref{fig:converseFailure}. We note that the index of $L[\cA]$ in $\Z^2$ is $2$. For any solution $(s_1,s_2) \in \V^\times(\cF)$, the point $(-s_1,-s_2)$ is also a solution. Hence, if $S \subseteq \V^\times(\cF)$ consists of pairs of the form $(s_1,s_2)$ and $(-s_1,-s_2)$, then the trace of $S$ is zero. In particular, the trace is constant under a perturbation of any of the coefficients and the sparse trace tests cannot recognize that $S$ might not be complete.
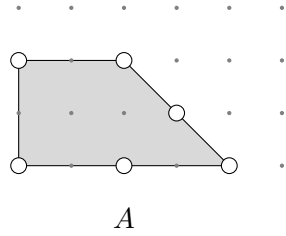
\begin{figure}[!htbp]
\begin{tikzpicture}[scale=.7]
\filldraw[fill=gray, fill opacity=0.3] (0,0) -- (4,0) -- (2,2) -- (0,2) -- cycle;
\foreach \i in {0,...,5} {\foreach \j in {0,...,3} {\filldraw[color=gray] (\i,\j) circle (.03);}}
\filldraw[fill=white] (0,0) circle (.15);
\filldraw[fill=white] (2,0) circle (.15);
\filldraw[fill=white] (4,0) circle(.15);
\filldraw[fill=white] (2,2) circle (.15);
\filldraw[fill=white] (3,1) circle (.15);
\filldraw[fill=white] (0,2) circle (.15);
\node at (2,-1) {$A$};
\end{tikzpicture}
\caption{Support $A$ such that the sparse trace tests cannot be applied on a polynomial system supported on $(A,A)$. }\label{fig:converseFailure}
\end{figure}
\end{example}

\section{Sparse resultants}
\label{sec:CombinatorialApproximationsOfTheSupport}
The efficient application of the sparse trace tests requires finding valid supports $\cB$ for input.  Lemma \ref{lem:ApproximationOfNecessary} provides simple candidates. It is a specialization of Theorem \ref{thm:combinatorialclassification}, proven here using the theory of sparse resultants.

\subsection{Supports and sparse resultants}

Let ${\mathcal A} = (A_1,\ldots,A_N)$ be a collection of finite subsets $A_i \subseteq \Z^n$. Note that $n$ and $N$ are not required to be equal. For $I \subseteq \mydef{[N]} \vcentcolon=\{1,2,\ldots,N\}$, we write $\mydef{\cA_I}$ for the subset $\{A_i\}_{i \in I}$.  The \mydef{defect} of a subset  $\cA_I \subseteq \cA$ is $\mydef{\defect(\cA_I)} \vcentcolon= \rank(\cA_I) - |I|$. When $n=N$, the following theorem of Minkowski characterizes when the mixed volume of $\cA$ is nonzero.

\begin{lemma}
\label{lem:Minkowski}
The mixed volume of $\cA=(A_1,\ldots,A_n)$ is nonzero if and only if $\defect(\cA_I)\geq 0$ for all $I \subseteq [n]$.
\end{lemma}

We use the following generalization of Minkowski's result in Lemma \ref{lem:WIdimension} of Section \ref{sec:sec:monodromyproof}.
\begin{lemma}
\label{lem:GeneralMinkowski}
For supports $\cA=(A_1,\ldots,A_k)$, $A_i \subseteq \Z^n$, and generic $\cF \in \C^{\cA}$, the dimension of $\V^\times(\cF)$ is $n-k$ when $\textrm{def}(\cA_I) \geq 0$ for all $I \subseteq [k]$. Otherwise, $\V^\times(\cF) = \emptyset$.
\end{lemma}

A collection $\cA$ is \mydef{essential} if its defect is $-1$ and every proper subset has nonnegative defect.  Let $\mydef{Q} = \{\textbf{0},e_1\} \subseteq \Z^n$.
When $n=N$ and $\MV(\cA)>0$, there is a unique subset $\mydef{\cA'}\subseteq \cA$ such that $\{Q\}\cup\cA'$ is essential {\cite{DAndreaJeronimo:2008,NewtonSturmfels:1994}}.

For a square polynomial system $\cF=\{f_1,\dots,f_N\}$ supported on $\cA$, we define $\mydef{Z(\cA)}$ to be the set of polynomial systems in $\C^{\cA}$ which have a solution in $(\C^\times)^{n}$. When $\cA$ has a unique essential subset, the Zariski closure of $Z(\cA)$ is a hypersurface in $\C^{\cA}$ defined by an irreducible polynomial in $\Z[u_{i,\alpha}]$, see \cite{NewtonSturmfels:1994}. In particular, this polynomial is unique, up to a nonzero constant. Following the terminology in \cite{Poisson}, we call this polynomial the \mydef{$\cA$-eliminant} or \mydef{sparse eliminant} and denote it by  \mydef{$\elim_\cA$}. If $\cA$ does not have a unique essential subset, then we define $\elim_\cA = 1$. 

Historically, the sparse eliminant has been referred to as the \emph{sparse resultant}. We use the following redefinition of the sparse resultant (as in \cite{Minimair}) which produces more uniform statements \cite{Poisson}.  When $\elim_\cA \neq 1$, the restriction of $\pi_\cA$ to $\pi^{-1}_{\cA}(Z(\cA))$ is generically \mydef{$d_{\cA}$}-to-one. We define the \mydef{$\cA$-resultant} or \mydef{sparse resultant} to be
$$\mydef{\res_{\cA}} = \elim_{\cA}^{d_{\cA}}.$$

For the universal polynomial system $\cF(u)$ over $\cA$, the sparse resultant $\mydef{\res_{\cA}(\cF(u))}$ is a polynomial in $\Z[u_{i,\alpha}]$.  Therefore, it can be evaluated at specific coefficients $\cF \in \C^{\cA}$. We write this in any of the following ways:
$$\res_{\cA}(\cF) = \res_{\cA}(f_0,\ldots,f_n) = \res_{\cA}(\{c_{i,\alpha}\}).$$

\subsection{Approximation of the necessary support}

Let $\cA$ be a square set of supports with $\MV(\cA)>0$.  In the hidden variable technique, we view $\cF(u)$ in the ring $\mathbb{C}[x_1][x_2,\dots,x_n]$ by treating $x_1$ as a coefficient. This turns $\cF(u)$ into a system of $n$ equations in $n-1$ variables, supported on the collection of projections $\pi(\cA)=(\pi(A_1),\ldots,\pi(A_n))$ under the  forgetful map $\mydef{\pi}:\Z^n\to\Z^{n-1}$ which forgets the first coordinate.  We let $\mathcal{G}(v)$ be the universal system over the support $\pi(\cA)$.

\begin{lemma}
\label{lem:SparseResultantExists}
If $\cA$ is a square set of supports and $\MV(\cA)>0$, then $\pi(\cA)$ contains a unique essential subset.
\end{lemma}
\begin{proof}
Since $\MV(\cA)>0$, the defect of any subset of $\cA$ is at least zero by Lemma \ref{lem:Minkowski}. On the other hand, since $|\cA| = n$, the defect of $\cA$ is at most zero and so $\rank(\cA)=n$. Since  $\rank(\cA) = n$, we must have that $\rank(\pi(\cA)) = n-1$ implying that $\pi(\cA)$ has defect $-1$. Thus, $\pi(\cA)$ contains at least one essential subset.

To see that this subset must be unique, suppose that $\pi(\cA_I)$ and $\pi(\cA_J)$ are two distinct essential subsets of $\pi(\cA)$. 
The defect of $\pi(\cA_I)$ is one less than the defect $\cA_I$ when $e_1$ is contained in the affine span of $\cA_I$ and equal to the defect of $\cA_I$ otherwise.
Since $\MV(\cA) > 0$, we conclude that $\textrm{def}(\cA_I) = \textrm{def}(\cA_J) = 0$ 
 and $e_1$ is in the affine span of both $\cA_I$ and $\cA_J$. This implies that $$\rank(\cA_{I\cup J}) \leq \rank(\cA_I)+\rank(\cA_J)-1 = |I|+|J|-1 < |I|+|J|.$$ But then $\cA_{I\cup J}$ has negative defect, contradicting the hypothesis that $\MV(\cA)> 0$.
\end{proof}
Lemma \ref{lem:SparseResultantExists} shows that $\res_{\pi(\cA)} \in \Z[v_{i,\beta}]$ is not $1$. Writing the polynomials $f_i$ of $\cF(u)$ as polynomials supported on $\pi(\cA)$ gives
$$f_i = \sum_{\beta \in \pi(A_i)} \mydef{h_{i,\beta}(x_1)} x^\beta \in \C[x_1][x_2,\ldots,x_n].$$ 
Thus, evaluating $\res_{\pi(\cA)}$ at the system $\cF(u)$ amounts to substituting $h_{i,\beta}(x_1)$ for $v_{i,\beta}$, that is,
\begin{equation}
\label{eq:hiddenVariableEquality}
\left.\res_{\pi(\cA)}(\mathcal G(v))\right|_{v_{i,\beta}=h_{i,\beta}(x_1)} = \res_{\pi(\cA)}(\cF(u))\in \Z[u_{i,\alpha}][x_1].
\end{equation}
This polynomial vanishes at all points $a_1$ such that the system $\cF(a_1,x_2,\ldots,x_n)$ has a solution in $(\C^\times)^{n-1}$.  The fact that $a_1$ may be zero is reflected in a power of $x_1$ appearing as a factor of this polynomial, as exhibited in the following lemma.
\begin{lemma}[{\cite[Proposition~4.7 and  Theorem~1.4]{Poisson}}]
\label{lem:hiddenPoisson}
Let $\cA$ be square with $\MV(\cA)>0$. Then there exists $d \in \Z$ such that, up to a constant,
$$\res_{\pi(\cA)}(\cF(u)) = x_1^{d}\res_{Q,A_1,\ldots,A_n}(z-x_1,f_1,\ldots,f_n)|_{z = x_1}.$$
Moreover, for generic $\cF \in \C^{\cA}$,
$$\res_{\pi(\cA)}(\cF) = x_1^d \prod_{s \in \V^\times(\cF)} (x_1-s_1).$$
\end{lemma}

\begin{lemma}[{\cite[Theorems 4 and 5]{ShadowRojas:1994}}]
As a polynomial in $x_1$, the degree of $\res_{\pi(\cA)}(\cF(u))$ is the mixed volume of the convex hulls of $A_i\cup(\{0\}\times\pi(A_i)).$ 
\end{lemma}
The convex hull of $A \cup (\{0\} \times \pi(A))$ is often called the \mydef{shadow} of the convex hull of $A$.
We define $\mydef{\res_{Q,\cA}^{(1)}(\cF(u))}$ to be the polynomial $\res_{Q,A_1,\ldots,A_n}(z-x_1,f_1,\ldots,f_n)|_{z = x_1}$ in $\Z[u][x_1]$. 
\begin{corollary}
As a polynomial in $x_1$, the degree of $\res_{Q,\cA}^{(1)}(\cF(u))$ is $\MV(\cA)$.
\label{cor:degreeHidden}
\end{corollary}

\begin{lemma}
\label{lem:easytrace}
Let $\cA=(A_1,\ldots,A_n)$ be a square collection of supports with $\MV(\cA)>0$. Let $\cF(u)$ be the universal polynomial system over $\cA$. We write the resultant $\res_{Q,\cA}^{(1)}(\cF(u))$ as $$\res_{Q,\cA}^{(1)}(\cF(u)) = q_{\MV(\cA)}(u)x_1^{\MV(\cA)} + q_{\MV(\cA)-1}(u)x_1^{\MV(\cA)-1}+\cdots + q_1(u) x_1 + q_0(u).$$Then
$$\Sigma_1(\V^{\times}(\cF(u))) = -\frac{q_{\MV(\cA)-1}(u)}{q_{\MV(\cA)}(u)}.$$
\end{lemma}
\begin{proof}
By Lemma \ref{lem:hiddenPoisson} and Corollary \ref{cor:degreeHidden}, we have that $\res_{Q,\cA}^{(1)}(\cF) = \prod_{s \in \V^{\times}(\cF)} (x_1-s_1)$
for generic $\cF = \cF(c) \in \C^{\cA}$, up to a constant. Since $\cF$ is generic, $q_{\MV(\cA)}(c) \neq 0$ and so the sum of the $x_1$ coordinates of points in $\V^{\times}(\cF)$ is $$\Sigma_1(\V^{\times}(\cF))=-\frac{q_{\MV(\cA)-1}(c)}{q_{\MV(\cA)}(c)}. $$ Since this equality holds for all generic $\cF \in \C^{\cA}$, the equality also holds for the universal polynomial system.
\end{proof}

\begin{remark}
D'Andrea and Jeronimo make the quotient in Lemma \ref{lem:easytrace} explicit. For a square collection of supports $\cA$ with positive mixed volume,  a specialization of formula \cite[Theorem 2.3]{DAndreaJeronimo:2008} gives
$$\Sigma_1(\V^\times(\mathcal F(u))) =  d_{Q,A_1,\ldots,A_n}\frac{\frac{\partial \elim_{Q,A_1,\ldots,A_n}}{\partial u_{0,e_1}}(1,f_1,\dots,f_n)}{\elim_{Q,A_1,\ldots,A_n}(1,f_1,\dots,f_n)} = \frac{\frac{\partial \res_{Q,A_1,\ldots,A_n}}{\partial u_{0,e_1}}(1,f_1,\ldots,f_n)}{\res_{Q,A_1,\ldots,A_n}(1,f_1,\ldots,f_n)}$$
where $\cF(u)=(f_0,\ldots,f_n)$ is the universal polynomial system over $(Q,A_1,\ldots,A_n)$.
\end{remark}

To prove the main result of this section, we first state an elementary result on the coefficients in the composition of a monomial with a collection of polynomials.
Let $\mydef{[x^k]f(x)}$ denote the coefficient of $x^k$ in $f(x)$.
\begin{lemma}
\label{lem:CoefficientsInComposition}
Define $g_i(x) = \sum_{j=0}^{d_i} c_{i,j}x^j \in \C[c][x]$ and fix $\gamma=(\gamma_1,\ldots,\gamma_n) \in \Z^n$. The polynomial $$g^\gamma(x) \vcentcolon= g_1(x)^{\gamma_1}\cdots g_n(x)^{\gamma_n}$$ has degree $d=d_1\gamma_1+\cdots+d_n\gamma_n$, and if $\prod c_{i_k,j_k}x^{d-\delta}$ is a term of $g^\gamma(x)$, then $\sum d_{i_k}-j_k =\delta$.

Consequently, if $d_i-j>\delta$, then $c_{i,j}$ does not appear in $[x^{d-\delta}]g^\gamma(x)$ and each term $\prod c_{i_k,j_k}x^{d-\delta}$ can have at most one $c_{i,j}$ with the property that $d_i-j > \delta/2$.
\end{lemma}

Before stating the main result of this section, we establish some notation. Fix a square collection of supports $\cA$ in $\Z^n$. For $\beta \in \pi(A_i)$, we let $(\mydef{m_{i,\beta}},\beta)$ be the integer point in the Newton polytope of $A_i$ with maximal first coordinate and $\mydef{\varepsilon_{i,\beta}} \in [0,1)$ be the distance in the direction of $e_1$ from $(\mydef{m_{i,\beta}},\beta)$ to the boundary of $A_i$. Note that there are finitely many $\varepsilon_{i,\beta}$, all of which are rational.  We define $\mydef{\lambda}$ to be the least common multiple of their denominators. Using $\mydef{\odot}$ to denote coordinate-wise multiplication, we define $\mydef{\widehat{\cA}}\vcentcolon=\Z^n\cap\conv(\lambda e_1 \odot \cA)\cap\pi^{-1}(\pi(\cA))$, so that $\lambda e_1 \odot \cA$ is $\cA$ scaled by $\lambda$ in the direction of $e_1$.  Note that $\pi(\cA)=\pi(\widehat{\cA})$ and all $\widehat{\varepsilon}_{i,\beta}$ for $\widehat{\cA}$ are zero, see Figure \ref{fig:StretchNotation}. 

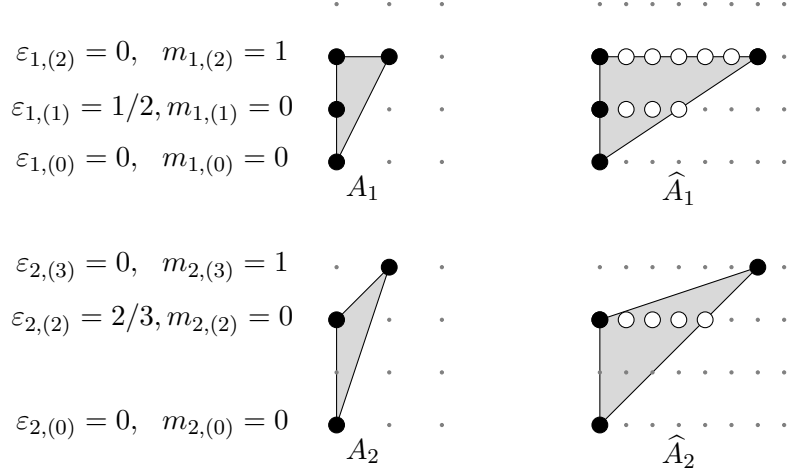
\begin{figure}
\begin{tikzpicture}[scale=.7]

\begin{scope}[shift={(0.5,0)}]
\filldraw[fill=gray, fill opacity=0.3] (0,0)  -- (1,2) -- (0,2) -- cycle;
\foreach \i in {0,...,2} {\foreach \j in {0,...,3} {\filldraw[color=gray] (\i,\j) circle (.03);}}
\filldraw[fill=black] (0,0) circle (.15);
\filldraw[fill=black] (0,1) circle (.15);
\filldraw[fill=black] (0,2) circle (.15);
\filldraw[fill=black] (1,2) circle (.15);
\end{scope}
\begin{scope}[shift={(0.5,-5)}]
\filldraw[fill=gray, fill opacity=0.3] (0,0) -- (1,3) -- (0,2) -- cycle;
\foreach \i in {0,...,2} {\foreach \j in {0,...,3} {\filldraw[color=gray] (\i,\j) circle (.03);}}
\filldraw[fill=black] (0,0) circle (.15);
\filldraw[fill=black] (1,3) circle (.15);
\filldraw[fill=black] (0,2) circle (.15);
\end{scope}
\begin{scope}[shift={(5.5,0)}]
\filldraw[fill=gray, fill opacity=0.3] (0,0) -- (3,2) -- (0,2) -- (0,2) -- cycle;
\foreach \i in {0,...,7} {\foreach \j in {0,...,3} {\filldraw[color=gray] (0.5*\i,\j) circle (.03);}}
\foreach \i in {0,...,3} {\filldraw[fill=white] (0.5*\i,1) circle (.15);}
\foreach \i in {0,...,6} {\filldraw[fill=white] (0.5*\i,2) circle (.15);}
\foreach \i in {0,...,0} {\filldraw[fill=black] (0.5*3*\i,1) circle (.15);}
\foreach \i in {0,...,1} {\filldraw[fill=black] (0.5*3*2*\i,2) circle (.15);}
\filldraw[fill=black] (0,0) circle (.15);
\end{scope}
\begin{scope}[shift={(5.5,-5)}]
\filldraw[fill=gray, fill opacity=0.3] (0,0) -- (3,3) -- (0,2) -- cycle;
\foreach \i in {0,...,7} {\foreach \j in {0,...,3} {\filldraw[color=gray] (0.5*\i,\j) circle (.03);}}
\filldraw[fill=black] (0,0) circle (.15);
\foreach \i in {0,...,4} {\filldraw[fill=white] (0.5*\i,2) circle (.15);}
\foreach \i in {0,...,0} {\filldraw[fill=black] (0.5*\i,2) circle (.15);}
\filldraw[fill=black] (3,3) circle (.15);
\end{scope}
\node at (1,-0.5) {$A_1$};
\node at (1,-5.5) {$A_2$};
\node at (7,-0.5) {$\widehat{A}_1$};
\node at (7,-5.5) {$\widehat{A}_2$};
\node at (-3,2) {$\varepsilon_{1,(2)} = 0,\,\,\,\, m_{1,(2)} = 1$};
\node at (-3,1) {$\varepsilon_{1,(1)} = 1/2, m_{1,(1)} = 0$};
\node at (-3,0) {$\varepsilon_{1,(0)} = 0,\,\,\,\, m_{1,(0)} = 0$};
\node at (-3,-2) {$\varepsilon_{2,(3)} = 0,\,\,\,\, m_{2,(3)} = 1$};
\node at (-3,-3) {$\varepsilon_{2,(2)} = 2/3, m_{2,(2)} = 0$};
\node at (-3,-5) {$\varepsilon_{2,(0)} = 0,\,\,\,\, m_{2,(0)} = 0$};
\end{tikzpicture}

\caption{Support $\cA=(A_1,A_2)$ and $\widehat \cA$ along with a depiction of $\varepsilon_{i,\beta}$ and $m_{i,\beta}$ for $\cA$.}
\label{fig:StretchNotation}
\end{figure}

\begin{theorem}\label{thm:combinatorialclassification}
Let $\cA=(A_1,\dots,A_n)$ be a collection of supports in $\Z^n$ such that $\MV(\cA)>0$. The support of $q_{\MV(\cA)-\delta}(u)$ is contained in $\offset(\cA,\delta)$. Moreover, the function $u \mapsto q_{\MV(\cA)-\delta}(u)$ is an affine linear function of the coefficients indexed by points in $\cA\backslash \offset(\cA,\delta/2)$.
\end{theorem}

\begin{proof}
We characterize which monomials in $\{u_{i,\alpha}\}_{\alpha \in A_i}$ may appear in the support of the coefficient 
$q_{\MV(\cA)-\delta}(u)$ of $x_1^{\MV(\cA)-\delta}$ in $\res^{(1)}_{Q,\cA}(\cF(u)) \in \mathbb{Z}[u][x_1]$.
Let $\mydef{R(x_1)} = \res_{\pi(\cA)}(\cG(h(x_1)))$ be the polynomial in $\Z[u][x_1]$ resulting from the substitution $v_{i,\beta} \mapsto h_{i,\beta}(x_1)$ in Equation \eqref{eq:hiddenVariableEquality}. Let $\mydef{d}$ be the degree of $R(x_1)$ in $x_1$.

Consider $\widehat \cA$ as defined prior to the statement of this lemma. We let $\mydef{\widehat{h_{i,\beta}}(x_1)}$, and $\mydef{\widehat{R}(x_1)}$ be the analogues of $h_{i,\beta}(x_1)$ and $R(x_1)$ for $\widehat \cA$, and $\widehat{d}$ the degree of $\widehat R(x_1)$. We note that $\res_{\pi(\cA)}(\cG(v))=\res_{\pi(\widehat{\cA})}(\cG(v))$.  Since $\mathcal S(\widehat \cA) = \lambda e_1 \odot \mathcal S(\cA)$ where $\mathcal S$ denotes the corresponding shadow polytope, we have that $\widehat{d}=\lambda d$.

Writing $h_{i,\beta}(x_1) = \sum_{j=0}^{m_{i,\beta}} c_{i,(j,\beta)}x_1^j$, we observe that $c_{i,(j,\beta)}$ appears in $[x_1^{d-\delta}]R(x_1)$ only if $\widehat{c}_{i,(\lambda j,\beta)}$ appears in $[x_1^{\lambda(d-\delta)}]\widehat R(x_1)$. Moreover, $[x_1^{\lambda(d-\delta)}]\widehat R(x_1)$ is the sum of $[x_1^{\lambda(d-\delta)}]\prod (\widehat h_{i,\beta}(x_1))^{\gamma_{i,\beta}}$ over all $\gamma$ indexing monomials $v^\gamma= \prod v_{i,\beta}^{\gamma_{i,\beta}}$ in $\res_{\pi(\cA)}(\cG(v))$. Since $\deg_{x_1}(\widehat h_{i,\beta}(x_1)) = \lambda(m_{i,\beta}+\varepsilon_{i,\beta})$, Lemma~\ref{lem:CoefficientsInComposition} implies that $\widehat c_{i,(\lambda j, \beta)}$ appears in $[x_1^{\lambda(d-\delta)}]$ only if 
$$\lambda(m_{i,\beta}+\varepsilon_{i,\beta}) - \lambda j \leq \widehat{d}-\lambda(d-\delta)=\lambda \delta.$$
Dividing through by $\lambda$ gives the necessary condition that
$$m_{i,\beta}+\varepsilon_{i,\beta} - j \leq \delta.$$
We observe that $m_{i,\beta}+\varepsilon_{i,\beta} - j$ is the distance in the first coordinate from the point $(j,\beta)$ to the boundary of the Newton polytope of $A_i$. Hence, $[x_1^{d-\delta}]R(x_1)$ may only involve coefficients $u_{i,\alpha}$ where the distance (in the direction of the first coordinate) from $\alpha$ to the boundary of the Newton polytope of $A_i$ is at most $\delta$, that is, $\offset(\cA,\delta)$. Since $R(x_1) = x_1^k\res_{Q,\cA}^{(1)}(\cF(u))$, we have that $[x_1^{d-\delta}]R(x_1) = [x_1^{\MV(\cA)-\delta}]\res_{Q,\cA}^{(1)}(\cF(u)) = q_{\MV(\cA)-\delta}(u)$.

 Moreover, each term of $[x_1^{d-\delta}]R(x_1)]$ can involve at most one coefficient $u_{i,\alpha}$ having this distance to the boundary greater than $\delta/2$. Hence, the function $u \mapsto [x_1^{d-\delta}]R(x_1)$ is an affine linear function of the coefficients indexed by $\cA\backslash \offset(\cA,\delta/2)$, completing the proof.
\end{proof}

\section{Sparse monodromy}
\label{sec:Monodromy}
There has been recent progress in the study of monodromy groups of sparse polynomial systems \cite{Esterov, EsterovLang,SottileYahl}. Esterov showed that there are (essentially) two properties of $\cA$ which can cause $G(\pi_{\cA})$ to fail to be the full symmetric group. These properties explain the need for restrictions on the input to the sparse trace tests, as we detail in Section \ref{sec:sec:LacunaryTriangular}. In Section \ref{sec:sec:monodromyproof}, we extend Esterov's  result to restrictions of the branched cover $\pi_{\cA}$, which completes the proof of correctness for the sparse trace tests.

\subsection{Lacunary and triangular supports}
\label{sec:sec:LacunaryTriangular} Throughout this section, we assume $\cA$ is a square set of supports in $\mathbb{Z}^n$. 
The following example illustrates why the condition $L[\cA] = \Z^n$ is necessary for the sparse trace tests. 

\begin{example}
\label{ex:LacunaryCounterExample}
Let $\cF(x_1,x_2)$ be any Bernstein-generic sparse polynomial system supported on $\cA$, where $A_1=\{(0,0),(1,0),(1,2)\}$ and $A_2=\{(0,0),(1,0),(0,2),(1,2)\}$, see Figure \ref{fig:LacunaryCounterExample}. Since every power of $x_2$ is even, $\cF$ may be written in the coordinates $(y_1,y_2)=(x_1,x_2^2)$ for some system $\cG(y_1,y_2)$. Thus, the zeros of $\cF$ in the torus have the following form:
$$\V^\times(\cF) = \{(a_1,b_1),(a_1,-b_1),(a_2,b_2),(a_2,-b_2)\}.$$
The map $(x_1,x_2) \stackrel{\phi}{\mapsto} (x_1,x_2^2)$ is a two-to-one map from $\V^\times(\cF)$ to $\V^\times(\cG)$. 

There are two types of proper nontrivial subsets $S$ of $\V^\times(\cF)$ for which the sparse trace tests erroneously succeed. When 
$S$ consists of a single fibre of $\phi$, the trace $\Sigma_2(S)$ is zero. When $S$ consists of a single point in each fibre of $\phi$ the trace $\Sigma_1(S)$  is half of $\Sigma_1(\V^\times(\cF))$. In both cases, if $\Sigma_i(\V^\times(\cF))$ 
moves affine linearly during the sparse trace test, then so does $\Sigma_i(S)$. 

\begin{figure}[!htbp]
\begin{tikzpicture}[scale=.7]
\filldraw[fill=gray, fill opacity=0.3] (0,0) -- (1,0) -- (1,2) -- cycle;
\foreach \i in {0,...,2} {\foreach \j in {0,...,3} {\filldraw[color=gray] (\i,\j) circle (.03);}}
\filldraw[fill=white] (0,0) circle (.15);
\filldraw[fill=white] (1,0) circle (.15);
\filldraw[fill=white] (1,2) circle (.15);
\begin{scope}[shift={(3.5,0)}]
\filldraw[fill=gray, fill opacity=0.3] (0,0) -- (1,0) -- (1,2) -- (0,2) -- cycle;
\foreach \i in {0,...,2} {\foreach \j in {0,...,3} {\filldraw[color=gray] (\i,\j) circle (.03);}}
\filldraw[fill=white] (0,0) circle (.15);
\filldraw[fill=white] (1,0) circle (.15);
\filldraw[fill=white] (1,2) circle (.15);
\filldraw[fill=white] (0,2) circle (.15);
\end{scope}
\begin{scope}[shift={(9,0)}]
\filldraw[fill=gray, fill opacity=0.3] (0,0) -- (1,0) -- (1,1) -- cycle;
\foreach \i in {0,...,2} {\foreach \j in {0,...,3} {\filldraw[color=gray] (\i,\j) circle (.03);}}
\filldraw[fill=white] (0,0) circle (.15);
\filldraw[fill=white] (1,0) circle (.15);
\filldraw[fill=white] (1,1) circle (.15);
\end{scope}
\begin{scope}[shift={(12.5,0)}]
\filldraw[fill=gray, fill opacity=0.3] (0,0) -- (1,0) -- (1,1) -- (0,1) -- cycle;
\foreach \i in {0,...,2} {\foreach \j in {0,...,3} {\filldraw[color=gray] (\i,\j) circle (.03);}}
\filldraw[fill=white] (0,0) circle (.15);
\filldraw[fill=white] (1,0) circle (.15);
\filldraw[fill=white] (1,1) circle (.15);
\filldraw[fill=white] (0,1) circle (.15);
\end{scope}
\node at (0.5,-0.5) {$A_1$};
\node at (4,-0.5) {$A_2$};
\node at (9.5,-0.5) {$B_1$};
\node at (13,-0.5) {$B_2$};
\end{tikzpicture}
\caption{Lacunary support $\cA=(A_1,A_2)$ and a lacunary reduction $\cB=(B_1,B_2)$ of $\cA$.}\label{fig:LacunaryCounterExample}
\end{figure}
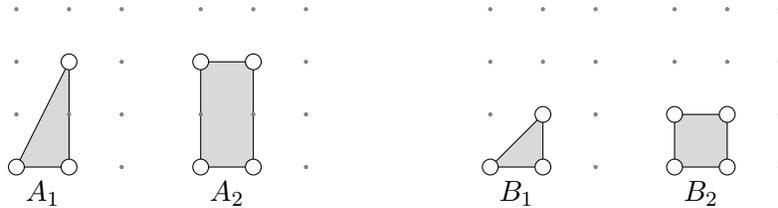

\end{example}

We now discuss the general case of the structure exhibited in Example \ref{ex:LacunaryCounterExample}.
\begin{definition}
We say a collection of supports $\cA$ is \mydef{lacunary} if $[\Z^n:L[\cA]]>1$. Similarly, any $\cF$ supported on lacunary support is called a \mydef{lacunary} system.  
\end{definition}

For any lacunary support $\cA$ in $\Z^n$, there exists a $\Z$-linear map ${\color{blue}{\Phi}}:\Z^n \to \Z^n$ and nonlacunary support $\cB$ such that $\Phi(\cB)=\cA$. We call such a $\cB$ a \mydef{lacunary reduction} of $\cA$.  Additionally, the sparse polynomial system $\cG$ in $\C^{\cB}$ with the same coefficient vector as $\cF \in \C^{\cA}$ is called a \mydef{lacunary reduction} of $\cF$ with respect to $\Phi$. The lacunary reduction has the property that
$$
\cG(y_1,\ldots,y_n) = \cG(\phi(x))=\cF(x)
$$
where $\phi:(\C^\times)^n \to (\C^\times)^n$ is the monomial map $\mydef{\phi(x_i)} = x^{\Phi(e_i)} \vcentcolon= x_1^{\Phi(e_i)_1}\cdots x_n^{\Phi(e_i)_n}$. The branched cover $\pi_{\cA}: X_{\cA} \to \C^{\cA}$ decomposes as $\pi_{\cA} = \pi_{\cB} \circ (\phi \times \id)$, where $\id(\cF) = \cG \in \C^{\cB}$.  The fibre of $\phi \times \id$ over a point $(y,\cG) \in X_{\cB}$ is 
$$(\phi \times \id)^{-1}(y,\cG)=\{(x,\cF)\,\, \mid \,\, \phi(x) = y, \,\, \cF=\cG\}.$$
This fibre has cardinality $\det(\Phi) = [\Z^n:L[\cA]]$ and is identified with the solutions to the binomial system $\phi(x)=y$. 

\begin{theorem}
\label{thm:lacunaryTraces}
Suppose $\cA$ is lacunary. 
\begin{itemize}
\item If $e_1 \notin L[\cA]$, then $\Sigma_1(\V^\times(\cF)) = 0$.
\item If $e_1 \in L[\cA]$, then $\Sigma_1(\V^\times(\cF)) = [\Z^n:L[\cA]] \Sigma_1(\V^\times(\cG))$ where $\Phi(\cB)=\cA$ is a lacunary reduction of $\cA$ satisfying $\Phi(e_1)=e_1$ and $\cG \in \C^{\cB}$ is the corresponding lacunary reduction of $\cF$.
\end{itemize}
\end{theorem}

\begin{proof}
If $e_1 \not\in L[\cA]$, then there exists a linear map $\Phi$ and lacunary reduction $\cB$ such that $\Phi(e_1) = ke_1$ for some $k>1$. Hence, if $(s_1,\ldots,s_n)$ is a solution to some generic $\cF \in \C^{\cA}$, then so is $(\omega_k \cdot s_1,s_2,\ldots,s_n)$ for any $k$-th root of unity $\omega_k$. Thus, the sum of the first coordinates of the solutions to $\cF$ is zero.

If $e_1 \in L[\cA]$, then $\Phi$ may be chosen so that $\Phi(e_1) = e_1$. Then the fibre over $\Phi$ of the point $(y,\cG)$  consists of $[\Z^n:L[\cA]]$ points with identical first coordinate $y_1$, and the result follows.
\end{proof}

The following corollary highlights how the application of our sparse trace tests on (invalid) lacunary support is one-sided by identifying nonempty proper subsets $S \subsetneq \V^\times(\cF)$ on which these algorithms return \texttt{pass}.

\begin{corollary}
\label{cor:TestFailsForLacunary}
Suppose $\cA$ is lacunary and $\cF$ is supported on $\cA$.
\begin{itemize}
\item If $e_1 \notin L[\cA]$, then the trace $\Sigma_1$ of a union of fibres over $\phi \times \id$ is zero.
\item If $e_1 \in L[\cA]$, then the trace $\Sigma_1$ of a union of $k$ points in each fibre is $\frac{k}{[\Z^n:L[\cA]]} \Sigma_1(\V^\times(\cF))$.  
\end{itemize}
\end{corollary}

We now illustrate a second property of $\cA$ which prevents the use of our sparse trace tests. This property never occurs for abundant $\cA$.

\begin{example}
\label{ex:TriangularCounterExample}
Let $\cF(x_1,x_2)=(f_1,f_2)$ be any Bernstein-generic sparse polynomial system supported on $\cA$, where $A_1=\{(0,0), (1,0), (2,0)\}$ and $A_2=\{(0,0),(1,0),(0,1),(2,1),(1,1),(0,2)\}$, see Figure \ref{fig:TriangularCounterExample}. Since the first coordinate of any point in $\V^\times(\cF)$ must be a solution to the square system $\cF_{1}(x_1)$, the zeros of $\V^\times(\cF)$ in the torus have the following form:
$$\V^\times(\cF) = \{(a_1,b_1),(a_1,c_1),(a_2,b_2),(a_2,c_2)\}.$$
The map $(x_1,x_2) \stackrel{\psi}{\mapsto} (x_1)$ is a two-to-one map from $\V^\times(\cF)$ to $\V^\times(\cF_1)$.

If $S$ consists of a single point in each fibre of $\psi$, then $\Sigma_1(S)$ is half of $\Sigma_1(\V^\times(\cF))$. As in the lacunary example, the trace $\Sigma_1(S)$ moves affine linearly whenever $\Sigma_1(\V^\times(\cF))$ does.

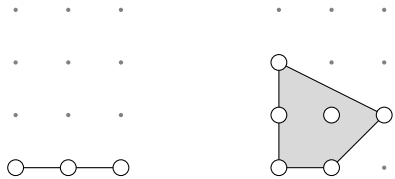
\begin{figure}[!htbp]
\begin{tikzpicture}[scale=.7]
\filldraw[fill=gray, fill opacity=0.3] (0,0) -- (1,0) -- (2,0) ;
\foreach \i in {0,...,2} {\foreach \j in {0,...,3} {\filldraw[color=gray] (\i,\j) circle (.03);}}
\filldraw[fill=white] (0,0) circle (.15);
\filldraw[fill=white] (1,0) circle (.15);
\filldraw[fill=white] (2,0) circle (.15);
\begin{scope}[shift={(5,0)}]
\filldraw[fill=gray, fill opacity=0.3] (0,0) -- (1,0) -- (2,1) -- (0,2) -- cycle;
\foreach \i in {0,...,2} {\foreach \j in {0,...,3} {\filldraw[color=gray] (\i,\j) circle (.03);}}
\filldraw[fill=white] (0,0) circle (.15);
\filldraw[fill=white] (1,0) circle (.15);
\filldraw[fill=white] (0,2) circle (.15);
\filldraw[fill=white] (1,1) circle (.15);
\filldraw[fill=white] (2,1) circle (.15);
\filldraw[fill=white] (0,1) circle (.15);
\end{scope}
\end{tikzpicture}
\caption{Triangular support $\cA=(A_1,A_2)$.}\label{fig:TriangularCounterExample}
\end{figure}

\end{example}

\begin{definition}
We say a collection of supports $\cA$ is \mydef{triangular} if there exists a proper subset $I \subsetneq [n]$ such that $\rank(\cA_I) = |I|$. Similarly, any $\cF$ supported on triangular support is called \mydef{triangular}. 
\end{definition}

We note that the condition that $\cB$ is abundant in our sparse trace tests implies that $\cA$ is not triangular.  When $\cA$ is triangular, witnessed by $I \subsetneq [n]$, the mixed volume of $\cA_I$ is defined to be its mixed volume within its affine span. When $1 < \MV(\cA_I) < \MV(\cA)$, we say that $\cA$ is \mydef{strictly triangular}.

Suppose that $\cF'$ is supported on the triangular support $\cA'$, witnessed by the subsystem $\cA'_I$. We consider the map $\Phi':\Z^{|I|} \to \Z^n$ which sends $e_{j_1},\ldots,e_{j_{|I|}}$ to generators of the saturated lattice $\textrm{span}(L[\cA'_I]) \cap \Z^n$, thus identifying $L[\cA'_I]$ with a sublattice of $\Z^{|I|}$. By choosing a complement of $L[\cA'_I]$ in $\Z^n$, we extend the map $\Phi'$ to an invertible change of coordinates $\Phi:\Z^n \to \Z^n$ and define 
$\phi:(\C^\times)^n \to (\C^\times)^n$ to be the corresponding invertible monomial map. 

We define $\cA \vcentcolon= \Phi^{-1}(\cA')$ and have the following isomorphism of incidence varieties:
$$X_{\cA'} \xrightarrow{\phi \times \Phi^{-1}} X_{\cA}.$$ We note that the system $\cF_I = \cF'_I(\phi(x))$ is a polynomial system in the variables $x_{j_1},\ldots,x_{j_{|I|}}$.

When $\cA$ is triangular, the map $\pi_{\cA}:X_{\cA} \to \C^{\cA}$ decomposes as
$$X_{\cA} \to X_{\cA_I} \times \C^{\cA_{I^c}} \to \C^{\cA},$$
where $I^c$ is the complement of $I$ in $[n]$. The first map takes $(x,\cF) \mapsto ((x_I,\cF_I),\cF_{I^c})$, and the second map takes $((x_I,\cF_I),\cF_{I^c}) \mapsto \cF$.

\begin{theorem}
\label{thm:triangularTraces}
Suppose $\cA$ is triangular, witnessed by $I \subsetneq [n]$. If $e_1 \in L[\cA_I]$ then $$\frac{\MV(\cA)}{\MV(\cA_I)}  \Sigma_1(\V^\times(\cF_I)) = \Sigma_1(\V^\times(\cF)),$$
where the mixed volume of $\cA_I$ is taken in $L[\cA_I]$.
\end{theorem}

\begin{proof}
If $e_1 \in L[\cA_I]$ then there exists an invertible $\Z$-linear map $\Phi:\Z^n \to \Z^n$ fixing $e_1$ such that $L[\Phi^{-1}(\cA_I)] \subseteq \langle e_1,\ldots,e_{|I|} \rangle$. 
Since this map fixes $e_1$, the corresponding monomial map fixes the first coordinates of points in $(\C^\times)^n$.  
Moreover, since $\Phi$ is invertible, it preserves mixed volumes. Hence, without loss of generality, we assume that $L[\cA_I] \subseteq \langle e_1,\ldots,e_{|I|} \rangle$.
Any fibre of $X_{\cA} \to X_{\cA_I} \times \C^{\cA_{I^c}}$ consists of $\frac{\MV(\cA)}{\MV(\cA_{I})}$-many points $\{(s_i,\cF)\}$ where each $s_i$ has the same first coordinate. Therefore, the trace $\Sigma_1(\V^\times(\cF))$ is $\frac{\MV(\cA)}{\MV(\cA_{I})}$ times the trace $\Sigma_1(\V^\times(\cF_I))$.
\end{proof}

\begin{corollary}
\label{cor:TestsFailForTriangular}
Suppose $\cA$ is triangular with $e_1 \in L[\cA]$ and $\cF$ is supported on $\cA$.  Then the trace of a union of $k$ points in each fibre of $X_{\cA} \to X_{\cA_I} \times \C^{\cA_{I^c}}$ is $\frac{k \cdot \MV(\cA_I)}{\MV(\cA)} \Sigma_1(\V^\times(\cF))$.  Thus, if $S$ is a union of $k$ points in each fibre, then the trace is an affine linear function of the coefficients of $\cF$, even if $S$ is not complete.
\end{corollary}

\begin{remark}
Theorem \ref{thm:triangularTraces} and Corollary \ref{cor:TestsFailForTriangular} may be thought of as partial triangular analogues to Theorem \ref{thm:lacunaryTraces} and Corollary \ref{cor:TestFailsForLacunary} for lacunary supports.
As depicted in Example \ref{ex:TriangularCounterExample}, triangularity guarantees only one type of proper subset of $\V^\times(\cF)$ which causes the sparse trace tests to erroneously succeed (compare to Example \ref{ex:LacunaryCounterExample}). Understanding whether additional problematic subsets of $\V^\times(\cF)$ always exist in the triangular setting is left to further research.
\end{remark}

Although the sparse trace tests may not be applied to lacunary or triangular supports, there are settings in which these properties are advantageous. In Section \ref{sec:Examples} we use these properties to more efficiently compute traces.   In \cite{SparseDecomposable}, the authors give a recursive algorithm for solving such polynomial systems. The only polynomial systems which need to be directly solved, in their method, are those which are neither lacunary nor triangular. Therefore, one may pair their work with our sparse trace tests to establish a method for verifying the completeness of solution sets to sparse systems, even when they are lacunary or triangular. 

\subsection{The restricted monodromy problem}
\label{sec:sec:monodromyproof}
The following result by Esterov shows that lacunary and triangular supports, are not only problematic for the sparse trace tests, but also restrict the monodromy group $G(\pi_{\cA})$.  
The goal of this section is to extend that result to restrictions of $\pi_{\cA}$.

\begin{proposition}{\cite[Theorem 1.5]{Esterov}}
\label{prop:Esterov} For a square set of supports $\cA$, the monodromy group $G(\pi_{\cA})$ is the full symmetric group if and only if either of the following are true: \begin{itemize} \item $\cA$ is neither lacunary nor strictly triangular.
\item  $\MV(\cA)=2$ and $[\Z^n:L[\cA]]=2$
\end{itemize}
\end{proposition}
\begin{remark}
The monodromy group associated to the family $u_{1,0}+u_{1,2}x^2$ of sparse polynomial systems is the full symmetric group, despite the support $\{0,2\}$ being lacunary. This example shows the necessity of the second condition above.
\end{remark}

Throughout this section, we assume that $\cA=(A_1,\ldots,A_n)$ is a collection of supports in $\Z^n$ with $\MV(\cA)>0$. We fix $\cB \subseteq \cA$, $\cC = \cA \backslash \cB$, and set $N=\sum_{i=1}^n |B_i|$. We take $\cF \in \C^{\cA}$ to be generic and write  $\cF=\cF_{\cB} + \cF_{\cC} \in \C^{\cB} \times \C^{\cC}$. We are interested in the monodromy action induced by varying the coefficients indexed by the points in $\cB$.

We consider restricted polynomial systems $\mydef{\cF_{u}} \vcentcolon = \cF(u_{\cB},\cF_{\cC})= \cF_{\cB}(u) + \cF_{\cC} = (f_{1,u},\ldots,f_{n,u})$, where 
$$f_{i,u}(x) = \underbrace{\sum_{\beta \in B_i}u_{i,\beta}x^\beta}_{g_{i,u}(x)} + \underbrace{\sum_{\gamma \in C_i} c_{i,\gamma}x^\gamma}_{h_{i}(x)},$$
and the corresponding restricted branched cover 
\[
 \xymatrix{&X_{\cB,\cF_{\cC}}=\{(x,b) \mid \cF_b(x)=0\} \ar[d]_{\pi_{\cB,\cF_{\cC}}}  \\
   &\hspace{1.25in}  \C^{\cB} \cong \C^{\cB} \times  \cF_{\cC} \subseteq \C^{\cA}. &
 }
\]
The linear space $\mathbb{C}^{\cB} \times \cF_{\cC}$ in $\mathbb{C}^{\cA}$ is not generic, and so, care must be taken in computing the~monodromy group. 
Following the approach of Harris \cite{Harris}, we establish conditions under which the monodromy group $G(\pi_{\cB,\cF_{\cC}})$ is the full symmetric group by showing that 
\begin{enumerate}
\item $G(\pi_{\cB,\cF_{\cC}})$ contains a simple transposition.
\item $G(\pi_{\cB,\cF_{\cC}})$ is $2$-transitive.
\end{enumerate}
Note that when $\cB=\cA$, the branched cover $\pi_{{\cA},\emptyset}$ is the same as $\pi_{\cA}$, so Proposition \ref{prop:Esterov} applies.

\begin{theorem}
\label{thm:SimpleTransposition}
If $\cA$ is a square set of supports which is neither lacunary nor triangular and $N>0$, then the monodromy group $G(\pi_{\cB,\cF_{\cC}})$ contains a simple transposition.
\end{theorem}
\begin{proof}
By Proposition \ref{prop:Esterov}, $\pi_{\cA}$ contains a simple transposition whenever $\cA$ is neither lacunary nor triangular. In this setting, the transposition is witnessed by any complex line in $\C^{\cA}$ which crosses the $\cA$-discriminant transversally. A monodromy loop, within this complex line, around its isolated point of intersection with the $\cA$-discriminant induces the transposition. Thus, it is enough to show that there exists such a line in $ \C^{\cB} \times \cF_{\cC}$ which also crosses the discriminant transversally. 

There are only two ways that a generic line in $\C^{\cB} \times \cF_{\cC}$ fails to cross the discriminant transversally: the discriminant does not involve the coefficients indexed by $\cB$ or its intersection with $\C^{\cB} \times \cF_{\cC}$ is singular everywhere. 

Since $\cA$ is neither lacunary nor triangular, Proposition \ref{prop:Esterov} implies $G(\pi_{\cA})$ contains a simple transposition and so 
the $\cA$-discriminant is not singular everywhere. Since $\bigcup_{\cG_{\cC} \in \C^{\cC}} \C^{\cB} \times \cG_{\cC} = \C^{\cA}$, the intersection of the discriminant with $\C^{\cB} \times \cF_{\cC}$ for generic $\cF_{\cC}$ is not singular everywhere. Moreover, by \cite[Lemma 1.20]{Esterov} and \cite[Lemma 3.9]{Esterov}, the defining equation of the discriminant has positive degree in  $c_{\alpha}$ for each $\alpha \in \cA$. In particular, this equation involves the coefficients indexed by $\cB$.
\end{proof}

Finding conditions which imply that $G(\pi_{\cB,\cF_{\cC}})$ is $2$-transitive is more involved, and we follow the approach of \cite{Harris}. 
We consider the following incidence correspondence of the fibre-square of $\pi_{\cB,\cF_{\cC}}$,
\[
 \xymatrix{&\mydef{Y}=\{(x,y,b) \mid \cF_b(x)=\cF_b(y)=0\} \ar[rd]_{\pi} \ar[ld]_{p} \\
  S & & \C^{\cB}
 }
\]
equipped with projections to the sets $\mydef{S} = ((\mathbb{C}^\times)^n)^2$ and $\C^{\cB}$. To show $\pi_{\cB,\cF_{\cC}}$ is $2$-transitive, we rely on the following elementary result.

\begin{proposition}
The monodromy group $\pi_{\cB,\cF_{\cC}}$ is $2$-transitive if and only if $Y$ has two components of top dimension.
\end{proposition}

The structure of the argument is as follows: When $\cB$ is abundant, we stratify the image $p(Y)$ of $Y$ into  subvarieties of $S$. We compute the dimension of each stratum and show that the fibre dimension is constant over each. We show that there are only two preimages of strata which have top dimension and these come from irreducible strata. Since $Y$ is the union of these preimages, we conclude that $Y$ has at most two components of top dimension.

Suppose $\cB$ is abundant and the support $\cA$ has been shifted so $\textbf{0} \in B_i$ for each $i \in [n]$.
For all $I \subseteq [n]$, we define
\begin{align*}
\mydef{V_I} &= \{(x,y) \in S \,\,\mid\,\, x^\alpha = y^\alpha \text{ for all } \alpha \in L[\cB_I]\}\\
\mydef{U_I} &= V_I \backslash \bigcup_{J \supseteq I} V_J \quad  \quad \quad \quad
\mydef{W_I} = U_I \cap p(Y).
\end{align*} 
Since $\cB$ is abundant, $\rank(L[\cB_I]) = n$ for any $I \neq \emptyset$. Thus, after an invertible monomial change of coordinates, $L[\cB_I] = \mydef{k_1} \Z \oplus \cdots \oplus \mydef{k_n}\Z$, where $k_1,\ldots,k_n$ are the \mydef{invariant factors} of $L[\cB_I]$. With respect to this lattice, the variety $V_I$ decomposes into components $\mydef{V_I(\omega)}$, one for each tuple $\omega = (\omega_1,\ldots,\omega_n)$ where $\omega_i$ is a $k_i$-th root of unity. Each component $V_I(\omega)$ is an $n$-dimensional irreducible variety defined via the explicit parametrization 
\begin{align*}
\mydef{\varphi_\omega}: (\C^\times)^n &\to S \\
(x_1,\ldots,x_n) &\mapsto (x_1,\ldots,x_n,\omega_1 x_1,\ldots,\omega_n x_n).
\end{align*}
Note that if $V_J(\omega) \cap V_I(\omega') \neq \emptyset$, then they are equal. Hence, the irreducible components of $U_I$ are components of $V_I$. That is, when $I \neq \emptyset$, the components of $U_I$ are of the form $V_I(\omega)$ for some (but not all) tuples of roots of unity $\omega$. Hence, we refer to the components of $U_I$ as $\mydef{U_I(\omega)}$. 
The only exception is $U_{\emptyset} = S \backslash \bigcup_{\emptyset\neq I} V_I$
\begin{lemma}
\label{lem:FibreDimension}
For any $s \in W_I$, the dimension of $p^{-1}(s)$ is $N-2n+|I|$. 
\end{lemma}
\begin{proof}
The fibre of $p$ over $s=(x,y)$ is the solution set to the $2n$-many \emph{affine} linear equations
\begin{align}
\label{eq:fibreEquations}
g_{1,u}(x) = -h_i(x),\quad g_{2,u}(x)=-h_i(x),\quad & \ldots, \quad g_{n,u}(x)=-h_n(x)\\
g_{1,u}(y) = -h_i(y), \quad g_{2,u}(y)=-h_i(y), \quad & \ldots, \quad g_{n,u}(y)=-h_n(y) \notag
\end{align}
in the variables $\{u_{i,\beta} \,\, |\,\, \beta \in B_i\}$. Since $s \in W_I \subseteq p(Y)$, the fibre over $s$ is nonempty, and the dimension of the fibre is given by the dimension of the kernel of the linear map $u \mapsto (g_{i,u}(x),g_{i,u}(y))_{i=1}^n$. Since each condition ($g_{i,u}(x)=0$ or $g_{i,u}(y)=0$) defining this kernel only involves the variables $\{u_{i,\beta}\}_{\beta \in B_i}$, the only linear dependencies amongst System \eqref{eq:fibreEquations} are between $g_{i,u}(x) = 0$ and $g_{i,u}(y)=0$ for some $i$. Since the $u_{i,\beta}$ are indeterminants, this linear dependency occurs exactly when the vectors $\{x^\beta\}_{\beta \in B_i}$ and $\{y^\beta\}_{\beta \in B_i}$ are linearly dependent. Equivalently, since $\textbf{0} \in B_i$ is a coordinate of each vector, they must be equal. Hence, the codimension of $p^{-1}(s)$ is  $2n-|I|$, and its dimension is $N-2n+|I|$.
\end{proof}

\begin{corollary}
\label{cor:WIasZeroSet}
The variety $W_I$ is $U_I \cap \V(\{h_i(x)-h_i(y)\}_{i \in I})$. In particular, $W_{\emptyset} = U_{\emptyset}$ has dimension $2n$ and $p$ is dominant.
\end{corollary}
\begin{proof}
A point $s=(x,y) \in U_I$ fails to be in $W_I$ whenever $p^{-1}(s)$ is empty, or equivalently, System~\eqref{eq:fibreEquations} is inconsistent. This occurs if and only if $g_{i,u}(x) = -h_i(x)$ and $g_{i,u}(y) = -h_i(y)$ describe parallel hyperplanes in the variables $\{u_{i,\beta}\}_{\beta \in B_i}$. This happens when $i \in I$ and $h_i(x) \neq h_i(y)$.
\end{proof}
We now introduce several definitions in preparation for establishing the dimensions of the $W_I(\omega)$ when $I \neq \emptyset$.  We assume that an invertible monomial change of coordinates has been applied so that $L[\cB_I] = k_1 \Z \oplus \cdots \oplus k_n \Z$. Since we are only interested in computing the dimension, we may perform this change of coordinates individually for each $I$. For a tuple of roots of unity $\omega$ indexing a component of $U_I$, we define the lattice 
$$\mydef{\mathcal L_{\omega}} = \{\alpha \in \Z^n \,\, |\,\,  \omega^\alpha =1\}$$
and the supports
$$\mydef{\cC^{\omega}} = (\mydef{C^\omega_1},\ldots,\mydef{C^\omega_n}) =(C_1 \backslash \mathcal L_\omega,\ldots,C_n \backslash \mathcal L_\omega)$$
We define $\mydef{\delta_I(\omega)}$ to be the number of supports in $\cC^\omega_I$ which are nonempty, that is the number of $i \in I$ such that $C^{\omega}_i \neq \emptyset$. 

\begin{example}
\label{ex:trickysupports}
Let $B_1 = \{(0,0),(1,0),(0,2)\}$, $B_2 = \{(0,0),(2,0),(0,2)\}$, $C_1 = \{(0,1),(1,1),(2,0)\}$, $C_2 = \{(1,1)\}$, and $\cA = \cB \cup \cC$ as depicted in Figure
\ref{fig:trickysupports}. We observe that $(0,0)$, $(1,1)$, and $(2,0)$ are all contained in $\mathcal L_{(-1,-1)}$ and so $\cC^{(-1,-1)} = (\{(0,1)\},\emptyset)$. Hence, $\delta_{2}(-1,-1) = 0$. On the other hand, we have $\mathcal C^{(1,-1)} = (\{(0,1),(1,1)\},\{(1,1)\})$ and $\delta_{12}(1,-1)=2$. The subset $\mathcal C^{(1,-1)}_2$ has negative defect.

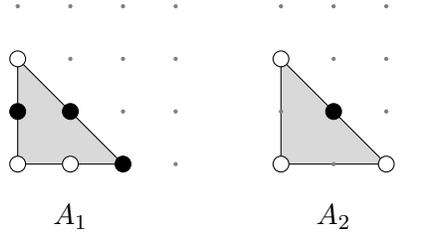
\begin{figure}[htpb!]
\begin{tikzpicture}[scale=.7]
\filldraw[fill=gray, fill opacity=0.3] (0,0) -- (2,0) -- (0,2) -- cycle;
\foreach \i in {0,...,3} {\foreach \j in {0,...,3} {\filldraw[color=gray] (\i,\j) circle (.03);}}
\filldraw[fill=white] (0,0) circle (.15);
\filldraw[fill=white] (1,0) circle(.15);
\filldraw (1,1) circle (.15);
\filldraw (0,1) circle (.15);
\filldraw (2,0) circle (.15);
\filldraw[fill=white] (0,2) circle (.15);

\node at (1,-1) {$A_1$}; 

\begin{scope}[shift={(5,0)}]

\node at (1,-1) {$A_2$}; 
\filldraw[fill=gray, fill opacity=0.3] (0,0) -- (2,0) -- (0,2) -- cycle;
\foreach \i in {0,...,3} {\foreach \j in {0,...,3} {\filldraw[color=gray] (\i,\j) circle (.03);}}
\filldraw[fill=white] (0,0) circle (.15);
\filldraw[fill=white] (0,2) circle(.15);
\filldraw (1,1) circle (.15);
\filldraw[fill=white] (2,0) circle (.15);
\end{scope}
\end{tikzpicture}
\caption{Support $\cA=(A_1,A_2)$ where the points in $\cB$ are white and the points in $\cC$ are filled.}
\label{fig:trickysupports}
\end{figure}
\end{example}

\begin{lemma}
\label{lem:WIdimension}
For $\emptyset \neq I \subseteq [n]$ and component $U_I(\omega)$, the variety $W_I(\omega)=U_I(\omega) \cap p(Y)$ is empty if the defect of any subset of $\cC^\omega_I$ is negative. Otherwise, $W_I(\omega)$ is the union of irreducible varieties of dimension $n-\delta_I(\omega)$. 
\end{lemma}
\begin{proof}
By Corollary \ref{cor:WIasZeroSet}, $W_I(\omega)$ may be identified with the zero set  $${\V(\{h_i(x_1,\ldots,x_n) - h_i(\omega_1x_1,\ldots,\omega_nx_n)\}_{i \in I})},$$ its preimage under the parametrization of $U_I(\omega)$. Explicitly, these polynomials are
\begin{equation}
\label{eq:initialSparsePullback}
\sum_{\gamma \in C_i} c_{i,\gamma}x^\gamma - \sum_{\gamma \in C_i} c_{i,\gamma}(\omega_1 x_1)^{\gamma_1} \cdots (\omega_nx_n)^{\gamma_n} \hspace{0.3 in} \text{ for } i \in I.
\end{equation}
After collecting terms, this is the sparse polynomial system 
\begin{equation}
\label{eq:sparsePullback}
\sum_{\gamma \in \cC_I^\omega} c_{i,\gamma}(1-\omega^{\gamma})x^{\gamma}  \hspace{ 0.3 in} \text{ for } i \in I.
\end{equation}
System \eqref{eq:initialSparsePullback} is supported on $C_i$, however, any term indexed by $\gamma \in C_i \cap \mathcal L_\omega$ does not appear after collecting terms to obtain System \eqref{eq:sparsePullback} because $(1-\omega^\gamma) =0$. Hence, System \eqref{eq:sparsePullback} is supported on $\cC_I^\omega$. The key observation is that System \eqref{eq:sparsePullback} is generic with respect to $\cC_I^\omega$: the nonzero constant $1-\omega^\gamma$ may be absorbed into the coefficient $c_{i,\gamma}$ by a reparametrization.  By Lemma \ref{lem:GeneralMinkowski}, the dimension of $W_I$ is $n-\delta_I(\omega)$ when the defect of every subset of $\cC^{\omega}_I$ is nonnegative and empty otherwise.
\end{proof}

\begin{theorem}
\label{thm:2Transitive}
Suppose $\cA$ is a square set of supports with $\MV(\cA)>0$. 
If $\cB \subseteq \cA$ is abundant, then the monodromy group $G(\pi_{\cB,\cF_{\cC}})$ is $2$-transitive if and only if $\cA$ is not lacunary.
\end{theorem}
\begin{proof}
The backwards direction follows from Proposition \ref{prop:Esterov}: If $\cA$ is lacunary, then $G(\pi_{\cA})$ is not $2$-transitive, so the restriction $G(\pi_{\cB,\cF_{\cC}})$ is not either. 

For the forward direction, we compute the dimension of the preimage $p^{-1}(W)$ of a component $W$ of $W_I(\omega)$ for $I \neq \emptyset$. The dimension of $W$ is $n-\delta_I(\omega)$ by Lemma \ref{lem:WIdimension}. The dimension of the fibre over a point in $W$ is  $N-2n+|I|$ by Lemma \ref{lem:FibreDimension}. Hence, the dimension of $p^{-1}(W)$ is $N-n+|I|-\delta_I(\omega)$ which only obtains the value of $N$ when $|I|=n$ and $\delta_I(\omega)=0$.

Suppose there exists $W_I(\omega)$ such that $I=[n]$, $\delta_I(\omega) = 0$, and $\omega \neq (1,\ldots,1)$. Since $\delta_I(\omega)=0$, each $C_i$ is contained in the lattice $\mathcal L_\omega$, a proper sublattice of $\Z^n$ since $\omega \neq (1,\ldots,1)$. Note that $\mathcal L_\omega$ also contains $\cB$ since $L[\cB] = k_1 \Z \oplus \cdots \oplus k_n \Z$ and the $\omega_i$ are $k_i$-th roots of unity. Thus, $\cC$ and $\cB$ are both contained in $\mathcal L_\omega$ and so $\cA = \cB \cup \cC$ must also be contained in $\mathcal L_\omega$. 
This is a contradiction since $\cA$ is not lacunary. Hence, the only $W_{[n]}(\omega)$ containing an irreducible component $W$ for which $\dim(p^{-1}(W))=N$ is the variety $W_{[n]}(1,\ldots,1) = \{(x,y) 
\in S \mid x=y\}$. Its preimage is the diagonal of $Y$ and is irreducible of dimension $N$.

The only other variety in the stratification of $p(Y)$ is $W_\emptyset$. This is an irreducible variety of dimension $2n$, and the dimension of a fibre of a point in $W_\emptyset$ is $N-2n$. Hence, $p^{-1}(W_\emptyset)$ is irreducible of dimension $N$. Thus, the only two components of $Y$ of dimension $N$ are $p^{-1}(W_\emptyset)$ and $p^{-1}(W_{[n]}(1,\ldots,1))$.
\end{proof}

Putting Theorem \ref{thm:SimpleTransposition} and Theorem \ref{thm:2Transitive} together, we arrive at our main result.
\begin{theorem}
\label{thm:monodromy}
Let $\cA$ be a nonlacunary square set of supports with $\MV(\cA)>0$. Suppose $\cB \subseteq \cA$ is abundant. Then the monodromy group $G(\pi_{\cB,\cF_{\cC}})$ is the full symmetric group.
\end{theorem}

Finally, using the machinery developed in this section, we provide a simple proof that the first coordinates of the solution set $\V^\times(\cF)$ are distinct whenever $\cA$ is neither lacunary nor triangular.

\begin{lemma}
\label{lem:distinctfirstcoordinates}
If $\cA$ is a square set of supports which is neither lacunary nor triangular and $\cF \in \C^{\cA}$ is generic, then the first coordinates of $\V^\times(\cF)$ are distinct. 
\end{lemma}
\begin{proof}
We let $\cB = \emptyset$ so that $Y$ is  the fibre-square of $\pi_{\cA}$. Since $G(\pi_{\cA})$ is $2$-transitive by Proposition~\ref{prop:Esterov}, $Y$ has two components of top dimension $N$. Independent of the argument in this section requiring $\cC$ to be abundant, these components are still $p^{-1}(W_{\emptyset})$ and the diagonal $p^{-1}(W_{[n]}(1,\ldots,1))$. Any other component of $Y$ cannot surject onto $\C^{\cA}$ via $\pi$ for dimension reasons. Hence, if ${s=(x,y,\cF)}$ is a point in a generic fibre over $\pi$ with $x_1=y_1$, we must have that $s \in p^{-1}(W_{[n]}(1,\ldots,1))$.
\end{proof}

\begin{example}
Consider the supports $\cA$ and $\cB$ in Example \ref{ex:trickysupports}. The inclusion lattice corresponding to the sets $\{V_I\}_{I \subseteq [2]}$ and $\{U_I\}_{I \subseteq [2]}$ are displayed in Figure \ref{fig:inclusionposets}. Hence, the set $S  = ((\C^\times)^2)^2$ is stratified by $U_{12}(1,1),U_{12}(1,-1),U_{2}(-1,-1),U_{2}(-1,1),$ and $U_{\emptyset}$. By Lemma \ref{lem:WIdimension}, $W_{12}(1,-1)$ and $W_{2}(-1,1)$ are both empty since both $\mathcal B^{(1,-1)}_{\{2\}}$ and $\mathcal B^{(-1,1)}_{\{1\}}$ have negative defects. Hence, the only nonempty components of $p(Y)$ are $W_{12}(1,1)$,$W_{2}(-1,-1)$, and $W_\emptyset$. Since $\delta_{12}(1,1)=0$ and $\delta_{2}(-1,-1) = 0$, we have that $W_I(\omega)=U_I(\omega)$.  The dimension counts for $p^{-1}(W_{12}(1,1)),p^{-1}(W_{2}(-1,-1)),$ and $p^{-1}(W_\emptyset)$ appear in Table \ref{table:DimensionCounts}.
We remark that $p^{-1}(W_2(-1,-1))$ is actually a subvariety of $p^{-1}(W_{\emptyset})$, and not its own component.

\begin{figure}[htpb!]
$$
\begin{tikzcd}[column sep=0.1cm,row sep=.3cm]
& V_{12}\arrow[dl]\arrow[dr] & \\
V_1 \arrow[dr]& & V_{2} \arrow[dl]\\
& V_\emptyset & 
\end{tikzcd} \hspace{0.2in}
\begin{tikzcd}[column sep=0.1cm,row sep=.3cm]
& \Z \oplus 2\Z & \\
\Z \oplus 2\Z \arrow[ur]& & 2\Z \oplus 2\Z \arrow[ul]\\
& \textbf{0} \arrow[ur]\arrow[ul] & 
\end{tikzcd} \hspace{0.2 in}
$$
\caption{The inclusion poset of $\{V_I\}_{I \subseteq \{1,2\}}$ and the inclusion poset of the corresponding lattices $\{L[\cB_I]\}_{I \subseteq \{1,2\}}$.}
\label{fig:inclusionposets}
\end{figure}
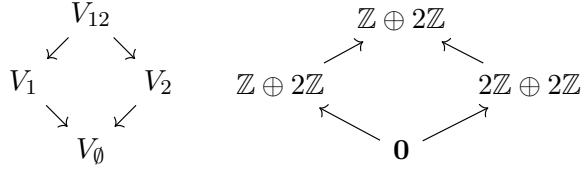

\begin{table}[htpb!]
\begin{tabular}{l|c|c|c||c}
Stratum & $\dim(U_I(\omega))$ & $\dim(W_I(\omega))$ & Fibre dimension & $\dim(p^{-1}(W_I(\omega))$ \\ \hline
$U_{12}(1,1)$ & $2$ & $2$ & $4$ & $6$ \\
$U_{12}(1,-1)$ & $2$ & $-1$ & $-1$ & $-1$ \\
$U_{2}(-1,1)$ & $2$ & $-1$ & $-1$ & $-1$ \\
$U_{2}(-1,-1)$ & $2$ & $2$ & $3$ & $5$ \\
$U_{\emptyset}$ & $4$ & $4$ & $2$ & $6$ 
\end{tabular}
\caption{Summary of dimension counts for various strata of $S$.}
\label{table:DimensionCounts}
\end{table}

\end{example}

\section{Examples and extensions}
\label{sec:Examples}
We collect a gallery of examples that illustrate and extend our theorems.

\subsection{Computing the trace by reducing to the necessary support}
Fix $\cA$ to be a square collection of supports with positive mixed volume, and $\cF \in \C^{\cA}$ to be Bernstein-generic. Then $\Sigma_1(\V^\times(\cF))$ is equal to $\Sigma_1(\V^\times(\cF_{\cN}))$ since, by definition, this trace depends only on the coefficients of $\cF$ in the necessary support $\cN \subseteq \cA$. 
The advantage of reducing to $\cF_{\cN}$ is that $\MV(\cN)$ is often significantly smaller than $\MV(\cA)$. Numerically, one may perform a direct computation of $\Sigma_1(\V^\times(\cF))$ by solving a system of much smaller degree. Algebraically, resultant computations on $\cF_{\cN}$ are likely to be much easier than resultant computations on $\cF$. 

\begin{example}\label{ex:Bezout}
Consider a polynomial system $\cF=(f_1,\ldots,f_n)$ in $n$ variables, where $\deg(f_i)=d_i$ for $i=1,\ldots,n$. Then $\cF$ is a Bernstein-generic sparse polynomial system supported on $$\mydef{\Delta} = (d_1\Delta_n,d_2\Delta_n,\ldots,d_n\Delta_n),$$
where $\mydef{\Delta_n}$ consists of the $n+1$ vertices of the standard $n$-dimensional simplex and $d\Delta_n$ consists of all lattice points in the $d$-th dilate of $\Delta_n$. 

A naive numerical computation of $\Sigma_1(\V^\times(\cF))$ involves computing $\prod_{i=1}^n d = \MV(\Delta)$-many solutions, by B\'ezout's theorem. However, by Lemma \ref{lem:ApproximationOfNecessary}, this trace only depends on the terms of each $f_i$ of degree $d_i$ and $d_i-1$. The mixed volume of $\cN$ is equal to $\prod_{i=1}^{n} d_i - \prod_{i=1}^{n} (d_i-1)$, considerably smaller than that of $\cA$. A root of $\V(\cF) \subseteq \C^n$ at the origin of multiplicity $\prod_{i=1}^n (d_i-1)$ accounts for the difference. 

For example, two generic bivariate polynomials $\{f_1,f_2\}$ of degree $5$ have $25$ common roots in the torus, whereas the system $\{g_1,g_2\}$ obtained by ignoring all terms of degree $3$ or less has only $9=25-16$ common solutions in the torus (see Figure \ref{fig:dense_deg5}). The remaining $16$ solutions are supported at the origin of $(\C^\times)^2$.

\end{example}

\begin{figure}[htb]
\begin{tikzpicture}[scale=.7]
\filldraw[fill=gray, fill opacity=0.3] (0,0) -- (5,0) -- (0,5) -- cycle;
\foreach \i in {0,...,6} {\foreach \j in {0,...,\i} {\filldraw[color=gray] ({6-\j},{6-\i+\j}) circle (.03);}}
\foreach \i in {0,...,3} {\foreach \j in {0,...,\i} {\filldraw[fill=white] (\j,{\i-\j}) circle (.15);}}
\foreach \j in {0,...,4} {\filldraw[fill=white] (\j,{4-\j}) circle (.15);\filldraw[pattern=crosshatch] (\j,{4-\j}) circle (.15);}
\foreach \j in {0,...,5} {\filldraw[fill=black] (\j,{5-\j}) circle (.15);}
\node at (3,-1) {(a)};
\begin{scope}[shift={(9,0)}]
\filldraw[fill=gray, fill opacity=0.3] (0,4) -- (4,0) -- (5,0) -- (0,5) -- cycle;
\foreach \i in {0,...,6} {\foreach \j in {0,...,\i} {\filldraw[color=gray] ({6-\j},{6-\i+\j}) circle (.03);}}
\foreach \i in {0,...,3} {\foreach \j in {0,...,\i} {\filldraw[color=gray] (\j,{\i-\j}) circle (.03);}}
\foreach \j in {0,...,4} {\filldraw[fill=white] (\j,{4-\j}) circle (.15);\filldraw[pattern=crosshatch] (\j,{4-\j}) circle (.15);}
\foreach \j in {0,...,5} {\filldraw[fill=black] (\j,{5-\j}) circle (.15);}
\node at (3,-1) {(b)};
\end{scope}
\end{tikzpicture}
\caption{(a) Support $5\Delta_2$.  (b)  Support $5\Delta_2 \backslash 3\Delta_2$. }\label{fig:dense_deg5}
\end{figure}
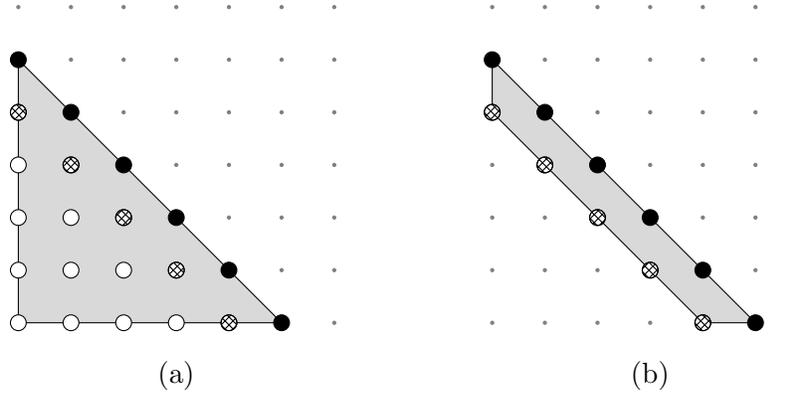

\begin{example}
Continuing Example \ref{ex:Bezout}, we compare timings of computing the traces of $\{f_1,f_2\}$ and $\{g_1,g_2\}$ by sampling random integers uniformly between $-10$ and $10$ for coefficients. Since the solutions to $\{g_1,g_2\}$ which are not in the torus are at the origin, they do not influence the trace, even though they affect the resultant.  We use the implementation in the Resultants package \cite{SparseResultants} in the Macaulay2 computer algebra system \cite{M2}.  The experiments were carried out on Clemson's Palmetto server on an Intel Xeon E5-2680 v3 CPU at 2.50GHz with 126 GB of ram and running CentOS linux.
\begin{table}[hbt]
\begin{tabular}{c|c|c|c|c|c}
Degree&5&10&15&20&25\\\hline
$\{f_1,f_2\}$&0.0301&0.931&13.7&98.5&494\\
$\{g_1,g_2\}$&0.0203&0.201&1.33&5.15&17.4
\end{tabular}
\medskip
\caption{Timings (in seconds) for computing the trace of two polynomial systems using the hidden variable resultant. One is a generic system $\{f_1,f_2\}$ of bivariate quintics. The other, $\{g_1,g_2\}$, is obtained from $\{f_1,f_2\}$ by ignoring terms of degree less than four.}
\end{table}
\end{example}

Rather than restricting the total degrees of each polynomial in a sparse polynomial system (as in Example \ref{ex:Bezout}), one can restrict the multidegrees of each polynomial. We illustrate our approach on such systems in the bivariate setting. 

\begin{example}\label{ex:rectangles}
Consider a Bernstein-generic bivariate system $\{f_1,f_2\}$ supported on $$\mydef{\rectangle} = (\rectangle_{k_1,\ell_1},\rectangle_{k_2,\ell_2})$$
where $\mydef{\rectangle_{k,\ell}} = ([0,k] \times [0,\ell]) \cap \Z^2$ (see Figure \ref{fig:rectangles}). The mixed volume of $\rectangle$ is $k_1\ell_2+k_2\ell_1$. The restricted polynomial system $\{g_1,g_2\}$, obtained by setting the coefficients of unnecessary monomials equal to zero, has support $(\rectangle_{1,\ell_1},\rectangle_{1,\ell_2})$ after translation. This support has mixed volume $\ell_1+\ell_2$. The translation of the supports corresponds to dividing each $g_i$ by $x_1^{k-1}$ to eliminate a component of $\V(g_1,g_2) \subseteq \C^2$ of multiplicity $(k-1)\ell_2+(k_2-1)\ell_1$ supported on $x_1=0$. Note that, generically, $\Sigma_2(\V^\times(f_1,f_2)) \neq \Sigma_2(\V^\times(g_1,g_2))$.

\begin{figure}[htb]
\begin{tikzpicture}[scale=.7]
\filldraw[fill=gray, fill opacity=0.3] (0,0) -- (5,0) -- (5,2) -- (0,2) -- cycle;
\foreach \i in {0,...,6} \filldraw[color=gray] (\i,3) circle (.03);
\foreach \i in {0,...,2} \filldraw[color=gray] (6,\i) circle (.03);
\foreach \i in {0,...,3} {\foreach \j in {0,...,2} {\filldraw[fill=white] (\i,\j) circle (.15);}}
\foreach \i in {0,...,2} {\filldraw[fill=white] (4,\i) circle (.15);\filldraw[pattern=crosshatch] (4,\i) circle (.15);}
\foreach \i in {0,...,2} {\filldraw[fill=black] (5,\i) circle (.15);}
\node at (3,-1) {(a)};
\begin{scope}[shift={(8,0)}]
\filldraw[fill=gray, fill opacity=0.3] (0,0) -- (2,0) -- (2,3) -- (0,3) -- cycle;
\foreach \i in {0,...,3} \filldraw[color=gray] (\i,4) circle (.03);
\foreach \i in {0,...,4} \filldraw[color=gray] (3,\i) circle (.03);
\foreach \j in {0,...,3} {\filldraw[fill=white] (0,\j) circle (.15);}
\foreach \i in {0,...,3} {\filldraw[fill=white] (1,\i) circle (.15);\filldraw[pattern=crosshatch] (1,\i) circle (.15);}
\foreach \i in {0,...,3} {\filldraw[fill=black] (2,\i) circle (.15);}
\node at (1.5,-1) {(b)};
\end{scope}
\begin{scope}[shift={(15,0)}]
\filldraw[fill=gray, fill opacity=0.3] (0,0) -- (1,0) -- (1,2) -- (0,2) -- cycle;
\foreach \i in {0,...,2} \filldraw[color=gray] (\i,3) circle (.03);
\foreach \i in {0,...,2} \filldraw[color=gray] (2,\i) circle (.03);
\foreach \i in {0,...,2} {\filldraw[fill=white] (0,\i) circle (.15);\filldraw[pattern=crosshatch] (0,\i) circle (.15);}
\foreach \i in {0,...,2} {\filldraw[fill=black] (1,\i) circle (.15);}
\node at (1,-1) {(c)};
\end{scope}
\begin{scope}[shift={(20,0)}]
\filldraw[fill=gray, fill opacity=0.3] (0,0) -- (1,0) -- (1,3) -- (0,3) -- cycle;
\foreach \i in {0,...,2} \filldraw[color=gray] (\i,4) circle (.03);
\foreach \i in {0,...,4} \filldraw[color=gray] (2,\i) circle (.03);
\foreach \i in {0,...,3} {\filldraw[fill=white] (0,\i) circle (.15);\filldraw[pattern=crosshatch] (0,\i) circle (.15);}
\foreach \i in {0,...,3} {\filldraw[fill=black] (1,\i) circle (.15);}
\node at (1,-1) {(d)};
\end{scope}
\end{tikzpicture}
\caption{Supports of $f_1,f_2,g_1,$ and $g_2$ as in Example \ref{ex:rectangles}.}
\label{fig:rectangles}
\end{figure}
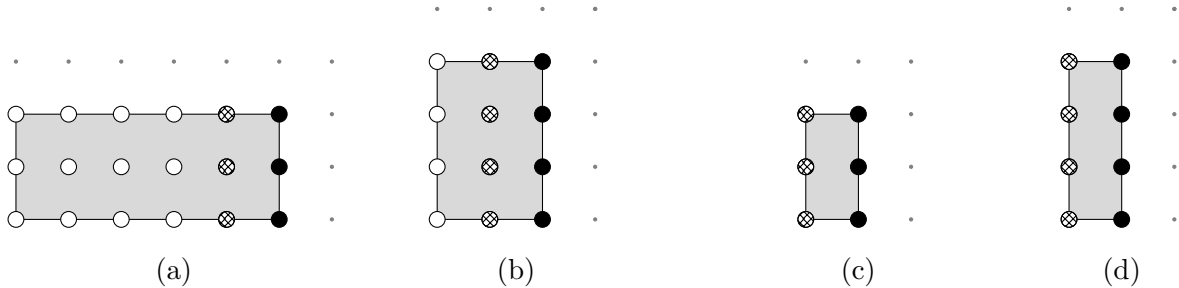

\end{example}

Even though the sparse trace tests cannot be directly applied to lacunary or triangular supports, one may use this special structure to compute traces more quickly (see Theorems \ref{thm:lacunaryTraces} and \ref{thm:triangularTraces}).

\begin{example}
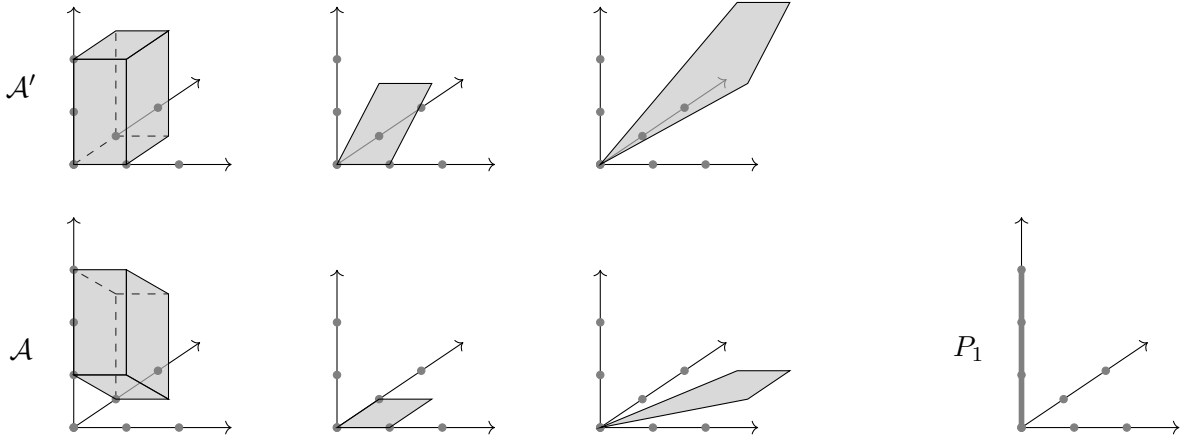
\begin{figure}
\begin{tikzpicture}[scale=.7]

\node at (-1,1.5) {$\cA'$};

\draw[dashed] (0,0) -- (0.8,.54);
\draw[dashed]  (0.8,.54) -- (1.8,0.54) ;
\draw[dashed]  (0.8,.54) -- (0.8,2.54) ;
\draw[->] (0,0) -- (0,3);
\draw[->] (0,0) -- (3,0);
\draw[gray] (0.8,.54) -- (2.25*.8,2.25*.54);
\draw[->] (2.25*0.8,2.25*.54) -- (3*0.8,3*.54);

\foreach \i in {0,...,2} \filldraw[color=gray] (\i,0) circle (.07);
\foreach \i in {0,...,2} \filldraw[color=gray] (0,\i) circle (.07);
\foreach \i in {0,...,2} \filldraw[color=gray] (\i*.8,\i*.54) circle (.07);

\filldraw[fill=gray, fill opacity=0.3] (0,0) -- (1,0) -- (1,2) -- (0,2) -- cycle;
\filldraw[fill=gray, fill opacity=0.3] (1,0)  -- (1.8,0.54) -- (1.8,2.54) -- (1,2) -- cycle;
\filldraw[fill=gray, fill opacity=0.3]  (1.8,2.54) -- (1,2) -- (0,2) -- (0.8,2.54) -- cycle;

\begin{scope}[shift={(5,0)}]

\draw[->] (0,0) -- (0,3);
\draw[->] (0,0) -- (3,0);
\draw[gray] (0,0) -- (1.9*.8,1.9*.54);
\draw[->] (1.9*0.8,1.9*.54) -- (3*0.8,3*.54);

\foreach \i in {0,...,2} \filldraw[color=gray] (\i,0) circle (.07);
\foreach \i in {0,...,2} \filldraw[color=gray] (0,\i) circle (.07);
\foreach \i in {0,...,2} \filldraw[color=gray] (\i*.8,\i*.54) circle (.07);

\filldraw[fill=gray, fill opacity=0.3] (0,0) -- (1,0) -- (1.8,1.54) -- (0.8,1.54) -- cycle;
\end{scope}

\begin{scope}[shift={(10,0)}]

\draw[->] (0,0) -- (0,3);
\draw[->] (0,0) -- (3,0);
\draw[gray] (0,0) -- (1.9*.8,1.9*.54);
\draw[gray, ->] (1.9*0.8,1.9*.54) -- (3*0.8,3*.54);

\foreach \i in {0,...,2} \filldraw[color=gray] (\i,0) circle (.07);
\foreach \i in {0,...,2} \filldraw[color=gray] (0,\i) circle (.07);
\foreach \i in {0,...,2} \filldraw[color=gray] (\i*.8,\i*.54) circle (.07);

\filldraw[fill=gray, fill opacity=0.3] (0,0) -- (2.8,1.54) -- (3.6,3.08) -- (2.6,3.08) -- cycle;
\end{scope}

\begin{scope}[shift={(0,-5)}]

\node at (-1,1.5) {$\cA$};

\draw[dashed] (0.8,2.54-1+1) -- (1.8,2.54-1+1);
\draw[dashed]  (0.8,2.54-1+1) -- (0,2+1) ;
\draw[dashed]  (0.8,.54-1+1) -- (0.8,2.54-1+1) ;
\draw[->] (0,0) -- (0,4);
\draw[->] (0,0) -- (3,0);
\draw[] (0,0) -- (.8,.54);
\draw[gray] (.8,.54) -- (2.25*.8,2.25*.54);
\draw[->] (2.25*0.8,2.25*.54) -- (3*0.8,3*.54);

\foreach \i in {0,...,2} \filldraw[color=gray] (\i,0) circle (.07);
\foreach \i in {0,...,3} \filldraw[color=gray] (0,\i) circle (.07);
\foreach \i in {0,...,2} \filldraw[color=gray] (\i*.8,\i*.54) circle (.07);

\filldraw[fill=gray, fill opacity=0.3] (0,1) -- (1,0+1) -- (1,2+1) -- (0,2+1) -- cycle;
\filldraw[fill=gray, fill opacity=0.3] (1,0+1)  -- (1.8,0.54-1+1) -- (1.8,2.54-1+1) -- (1,2+1) -- cycle;
\filldraw[fill=gray, fill opacity=0.3]  (0,0+1) -- (1,0+1) -- (1.8,0.54-1+1) -- (0.8,.54-1+1) -- cycle;
\end{scope}

\begin{scope}[shift={(5,-5)}]

\draw[->] (0,0) -- (0,3);
\draw[->] (0,0) -- (3,0);
\draw[] (0,0) -- (1.9*.8,1.9*.54);
\draw[->] (1.9*0.8,1.9*.54) -- (3*0.8,3*.54);

\foreach \i in {0,...,2} \filldraw[color=gray] (\i,0) circle (.07);
\foreach \i in {0,...,2} \filldraw[color=gray] (0,\i) circle (.07);
\foreach \i in {0,...,2} \filldraw[color=gray] (\i*.8,\i*.54) circle (.07);

\filldraw[fill=gray, fill opacity=0.3] (0,0) -- (1,0) -- (1.8,0.54) -- (0.8,0.54) -- cycle;
\end{scope}

\begin{scope}[shift={(10,-5)}]

\draw[->] (0,0) -- (0,3);
\draw[->] (0,0) -- (3,0);
\draw[] (0,0) -- (1.9*.8,1.9*.54);
\draw[->] (1.9*0.8,1.9*.54) -- (3*0.8,3*.54);

\foreach \i in {0,...,2} \filldraw[color=gray] (\i,0) circle (.07);
\foreach \i in {0,...,2} \filldraw[color=gray] (0,\i) circle (.07);
\foreach \i in {0,...,2} \filldraw[color=gray] (\i*.8,\i*.54) circle (.07);

\filldraw[fill=gray, fill opacity=0.3] (0,0) -- (2.8,.54) -- (3.6,1.08) -- (2.6,1.08) -- cycle;
\end{scope}

\begin{scope}[shift={(18,-5)}]

\node at (-1,1.5) {$P_1$};

\draw[->] (0,0) -- (0,4);
\draw[->] (0,0) -- (3,0);
\draw[] (0,0) -- (1.9*.8,1.9*.54);
\draw[->] (1.9*0.8,1.9*.54) -- (3*0.8,3*.54);

\foreach \i in {0,...,2} \filldraw[color=gray] (\i,0) circle (.07);
\foreach \i in {0,...,3} \filldraw[color=gray] (0,\i) circle (.07);
\foreach \i in {0,...,2} \filldraw[color=gray] (\i*.8,\i*.54) circle (.07);

\draw[line width=0.75mm, gray] (0,0) -- (0,3);
\end{scope}

\end{tikzpicture}
\caption{The Newton polytopes of $\cF'$ in the variables $(x_1,x_2,x_3)$ along with the Newton polytopes of $\cF$ in the variables $(x_1,x_2,x_2x_3) = (y_1,y_2,y_3)$ (after multiplying $f_1$ by $y_3$ to clear denominators). After solving the subsystem $\{f_2,f_3\}$ in $(y_1,y_2)$, back-substitution into $f_1$ involves solving a univariate cubic in $y_3$ supported on $P_1$.}
\label{fig:TriangularTraceExample}

\end{figure}

Consider a Bernstein-generic sparse polynomial system $\cF' = (f'_1,f'_2,f'_3)$ supported on $\cA'=(A'_1,A'_2,A'_3)$ consisting of all of the lattice points contained in the Newton polytopes displayed in Figure \ref{fig:TriangularTraceExample}. The system $\cF'$ is triangular, witnessed by the subset $\{2,3\}$. Let $\cA$ be the preimage of $\cA'$ under the linear map $\Phi(\alpha_1,\alpha_2,\alpha_3) = (\alpha_1,\alpha_2+\alpha_3,\alpha_3)$. If $\cF=(f_1,f_2,f_3)$ is the corresponding polynomial system supported on $\cA$, then  $\cF(y_1,y_2,y_3) = \cF(x_1,x_2x_3,x_3) = \cF'(x_1,x_2,x_3)$. The system $\cF$ is obviously triangular since $\{f_2,f_3\}$ is a square system in $y_1$ and $y_2$. 

Since $e_1 \in L[A_{\{2,3\}}]$ and $\Phi$ was chosen to fix $e_1$, the proof of Theorem \ref{thm:triangularTraces} implies that $\Sigma_1(\V^\times(\cF))$ is a multiple of $\Sigma_1\left(\V^\times\left(\cF_{\{2,3\}}\right)\right)$. Namely, $\Sigma_1(\V^\times(\cF))$ is three times the trace of the subsystem since {each} solution $(s_1,s_2) \in \V^\times\left(\cF_{\{2,3\}}\right)$ extends to three solutions to $\V^\times(\cF)$ ($f_1(s_1,s_2,y_3)$ is a univariate polynomial of degree three).  Geometrically, this is illustrated in Figure \ref{fig:TriangularTraceExample}: as the projection $P_1$ of $A_1$ onto ${L\left[\cA_{\{2,3\}}\right]^\perp}$ is a line segment of length three. Therefore, the trace of $\V^\times(\cF')$ can be calculated directly from the trace of $\V^\times(f_2,f_3)$ by computing the length of the projection of $A'_1$ onto a complement of $L\left[{\cA'}_{\{2,3\}}\right]$.

\end{example}

\subsection{Applications to nongeneric systems}

Suppose $\cF \in \C^{\cA}$ is not a Bernstein-generic system, but has $d < \MV(\cA)$ isolated solutions in the torus. The trace $\Sigma_1(\V^\times(\cF))$ is still a ratio of coefficients of $\res_{Q,\cA}^{(1)}(\cF)$. Since each $q_{k}(\cF)=[x_1^k]\res_{Q,\cA}^{(1)}(\cF)$ is an affine linear function of the set of coefficients given in Theorem \ref{thm:combinatorialclassification}, this trace may be computed just as in the previous section. However, the set of unnecessary coefficients may be smaller in these Bernstein-deficient cases.
\begin{example}\label{ex:notBernstein}
Suppose that $\cF=(f_1,f_2)$ consists of a pair of dense bivariate polynomials of degree $5$ which have $24$ solutions in $\C^2$, all of which are in the torus. Consequently, $\cF$ is not Bernstein-generic: there is one solution at infinity. Using the notation in Theorem \ref{thm:combinatorialclassification}, the trace $\Sigma_1(\V^\times(\cF))$ is $\frac{-q_{23}(\cF)}{q_{24}(\cF)}$. By Theorem \ref{thm:combinatorialclassification}, this is an affine linear function of the coefficients in $\cA \backslash \offset(\cA,1)$, expressed in Figure \ref{fig:not_generic} by the white and cross-hatched points in $\Z^2$. 
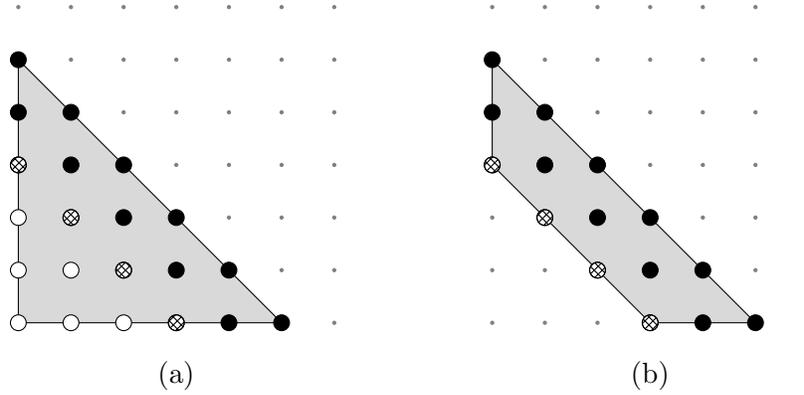
\begin{figure}[htb]
\begin{tikzpicture}[scale=.7]
\filldraw[fill=gray, fill opacity=0.3] (0,0) -- (5,0) -- (0,5) -- cycle;
\foreach \i in {0,...,6} {\foreach \j in {0,...,\i} {\filldraw[color=gray] ({6-\j},{6-\i+\j}) circle (.03);}}
\foreach \i in {0,...,2} {\foreach \j in {0,...,\i} {\filldraw[fill=white] (\j,{\i-\j}) circle (.15);}}
\foreach \j in {0,...,3} {\filldraw[fill=white] (\j,{3-\j}) circle (.15);\filldraw[pattern=crosshatch] (\j,{3-\j}) circle (.15);}
\foreach \j in {0,...,4} {\filldraw[fill=black] (\j,{4-\j}) circle (.15);}
\foreach \j in {0,...,5} {\filldraw[fill=black] (\j,{5-\j}) circle (.15);}
\node at (3,-1) {(a)};
\begin{scope}[shift={(9,0)}]
\filldraw[fill=gray, fill opacity=0.3] (0,3) -- (3,0) -- (5,0) -- (0,5) -- cycle;
\foreach \i in {0,...,6} {\foreach \j in {0,...,\i} {\filldraw[color=gray] ({6-\j},{6-\i+\j}) circle (.03);}}
\foreach \i in {0,...,2} {\foreach \j in {0,...,\i} {\filldraw[color=gray] (\j,{\i-\j}) circle (.03);}}
\foreach \j in {0,...,3} {\filldraw[fill=white] (\j,{3-\j}) circle (.15);\filldraw[pattern=crosshatch] (\j,{3-\j}) circle (.15);}
\foreach \j in {0,...,4} {\filldraw[fill=black] (\j,{4-\j}) circle (.15);}
\foreach \j in {0,...,5} {\filldraw[fill=black] (\j,{5-\j}) circle (.15);}
\node at (3,-1) {(b)};
\end{scope}
\end{tikzpicture}
\caption{(a) Support of $f_1$ and $f_2$  in Example \ref{ex:notBernstein}. These polynomials are two bivariate quintics with one common solution at infinity and $24$ common solutions in $(\C^\times)^2$.  (b) Support of the polynomials $f_1$ and $f_2$ after truncating the terms of degree less than $3$.}\label{fig:not_generic}
\end{figure}

As in Example \ref{ex:Bezout}, setting the unnecessary coefficients equal to zero produces a sparse polynomial system $\cG$ with $\Sigma_1(\V^\times(\cF)) = \Sigma_1(\V^\times(\cG))$, even though $|\V^\times(\cG)|=15<24 = |\V^\times(\cF)|$. Each polynomial in $\cG$ is supported on the second set of monomials in Figure \ref{fig:not_generic}. 
\end{example}

\subsection{Computing other traces}
Rather than computing the trace $\sum_{x \in \V^\times(\cF)} x_1 = \Sigma_1(\V^\times(\cF))$, one may be interested in computing the sum of a different monomial $x^\gamma$ over the solutions $\V^\times(\cF)$. Such a function is also called a \emph{trace} \cite{DAndreaJeronimo:2008}. The trace of $x^\gamma$ may be computed by applying an invertible monomial change of coordinates $(\C_x^\times)^2 \to (\C_y^\times)^2$ which identifies $x^\gamma$ with $y_1$. In the coordinates $y$, our results apply to the resulting polynomial system $\cG(y)$. As the support $\cA$ changes under this monomial change of coordinates, so do the sets $\offset(\cA,\delta)$. Consequently, we can find subsets of coefficients for which the function $\cF \mapsto \sum_{x \in \V^\times(\cF)} x^\gamma$ is affine linear. 

\begin{example}\label{ex:shear}

\begin{figure}[htbp!]
\begin{tikzpicture}[scale=.5]
\filldraw[fill=gray, fill opacity=0.3] (0,0) -- (5,-10) -- (0,5) -- cycle;
\foreach \i in {0,...,1} {\foreach \j in {0,...,\i} {\filldraw[fill=white] (\j,{\i-3*\j}) circle (.15);}}
\foreach \j in {0,...,2} {\filldraw[fill=white] (\j,{2-3*\j}) circle (.15);\filldraw[pattern=crosshatch] (\j,{2-3*\j}) circle (.15);}
\foreach \j in {0,...,3} {\filldraw[fill=white] (\j,{3-3*\j}) circle (.15);\filldraw[pattern=crosshatch] (\j,{3-3*\j}) circle (.15);}
\foreach \j in {0,...,4} {\filldraw[fill=white] (\j,{4-3*\j}) circle (.15);\filldraw[fill=black] (\j,{4-3*\j}) circle (.15);}
\foreach \j in {0,...,5} {\filldraw[fill=black] (\j,{5-3*\j}) circle (.15);}
\node at (3,-11) {(a)};
\begin{scope}[shift={(9,-3)}]
\filldraw[fill=gray, fill opacity=0.3] (0,0) -- (5,0) -- (0,5) -- cycle;
\foreach \i in {0,...,6} {\foreach \j in {0,...,\i} {\filldraw[color=gray] ({6-\j},{6-\i+\j}) circle (.03);}}
\foreach \i in {0,...,2} {\foreach \j in {0,...,\i} {\filldraw[fill=white] (\j,{\i-\j}) circle (.15);}}
\foreach \j in {0,...,3} {\filldraw[fill=white] (\j,{3-\j}) circle (.15);\filldraw[pattern=crosshatch] (\j,{3-\j}) circle (.15);}
\foreach \j in {0,...,4} {\filldraw[fill=white] (\j,{4-\j}) circle (.15);\filldraw[fill=black] (\j,{4-\j}) circle (.15);}
\foreach \j in {0,...,2} {\filldraw[fill=white] (\j,{2-\j}) circle (.15);\filldraw[pattern=crosshatch] (\j,{2-\j}) circle (.15);}
\foreach \j in {0,...,5} {\filldraw[fill=black] (\j,{5-\j}) circle (.15);}
\node at (3,-1) {(b)};
\end{scope}
\end{tikzpicture}
\caption{(a) Support for the pair of dense polynomials of degree $5$ in Example \ref{ex:shear} after a change of coordinates.  (b) The support for this pair of dense polynomials of degree $5$ shifted back to the original coordinates.}\label{fig:changeofvariables}
\end{figure}
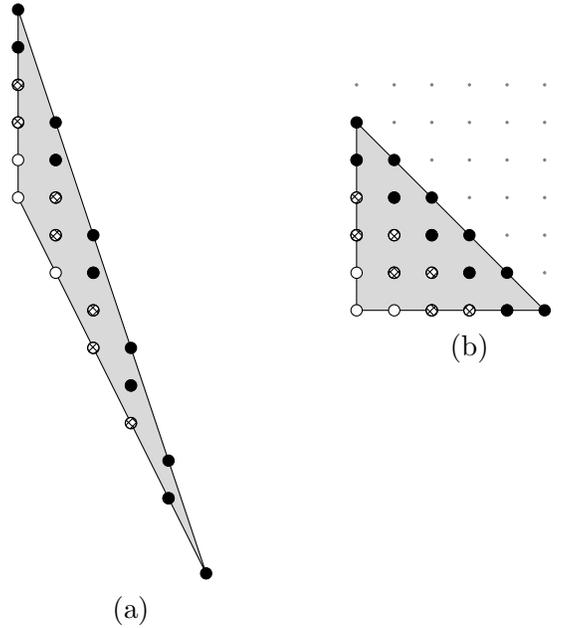

Consider the trace of $x_1x_2^2$ over the solutions $\V^\times(\cF)$ of a Bernstein-generic polynomial system $\cF=(f_1,f_2)$ supported on $\cA=(5\Delta_2,5\Delta_2)$. This trace may be computed by applying the monomial substitution $\cF(x_1,x_2) \mapsto \cF(y_1y_2^{-2},y_2)=\cG(y_1,y_2)$, which induces a map $\V^{\times}(\cF) \to \V^{\times}(\cG)$ on varieties which identifies the trace of $x_1x_2^2$ over $\V^\times(\cF)$ with the trace of $y_1$ over $\V^\times(\cG)$. The support of $\cG$ is displayed in Figure \ref{fig:changeofvariables}, where our results imply that the white and cross-hatched coefficients influence $\Sigma_1(\V^\times(\cG))$ affine linearly. Consequently, the same coefficients influence the trace of $x_1x_2^2$ over $\V^\times(\cG)$ affine linearly. In particular, the white points in Figure \ref{fig:changeofvariables} correspond to coefficients which do not influence $\sum_{x \in \V^\times(\cF)}x_1x_2^2$. 
\end{example}

\section*{Acknowledgements}
We would like to thank Carlos D'Andrea, Gabriella Jeronimo, and Alexander Esterov for helpful discussions. The work began while the authors were in residence at ICERM's semester on Nonlinear Algebra, and continued while the first author was at Texas A\&M University and the Max Planck Institute for Mathematics in the Sciences in Leipzig, Germany.

\bibliographystyle{plain}
\bibliography{SparseTraceTests}
\end{document}